\documentclass[11pt,letterpaper]{amsart}
\usepackage{amsthm}
\usepackage{amstext}
\usepackage{enumerate,enumitem}
\usepackage{amssymb}
\usepackage{amsmath,calligra,mathrsfs}
\usepackage[all,cmtip]{xy}
\usepackage{dsfont}
\usepackage{hyperref}
\usepackage{xcolor}

\usepackage{cite}
\makeatletter
\newcommand{\citecomment}[2][]{\citen{#2}#1\citevar}
\newcommand{\citeone}[1]{\citecomment{#1}}
\newcommand{\citetwo}[2][]{\citecomment[,~#1]{#2}}
\newcommand{\citevar}{\@ifnextchar\bgroup{;~\citeone}{\@ifnextchar[{;~\citetwo}{]}}}
\newcommand{\citefirst}{\@ifnextchar\bgroup{\citeone}{\@ifnextchar[{\citetwo}{]}}}
\newcommand{\cites}{[\citefirst}
\makeatother

\theoremstyle{plain}
\newtheorem{thm}{Theorem}[section]
\newtheorem{lem}[thm]{Lemma}
\newtheorem{prop}[thm]{Proposition}
\newtheorem{cor}[thm]{Corollary}
\theoremstyle{definition}
\newtheorem{defn}[thm]{Definition}

\theoremstyle{remark}
\newtheorem{rem}[thm]{Remark}
\theoremstyle{plain}
\newtheorem{assumption}[thm]{Assumption}

\DeclareMathOperator{\chara}{char}
\DeclareMathOperator{\Spec}{Spec}

\DeclareMathOperator{\Cl}{Cl}
\DeclareMathOperator{\Tr}{Tr}
\DeclareMathOperator{\Supp}{Supp}
\DeclareMathOperator{\codim}{codim}
\DeclareMathOperator{\Pic}{Pic}
\DeclareMathOperator{\Bl}{Bl}
\DeclareMathOperator{\Sing}{Sing}
\DeclareMathOperator{\NS}{NS}
\DeclareMathOperator{\Jac}{Jac}

\DeclareMathOperator{\Clns}{Cl^{ns}}
\DeclareMathOperator{\Bs}{Bs}
\def\Im{\mathrm{Im}}
\let\phi\varphi

\newcommand{\defi}[1]{\textsf{#1}} % for defined terms

\author{Lena Ji}
\address{1855 East Hall\\
530 Church Street\\
Department of Mathematics\\
University of Michigan\\
	Ann Arbor, MI 48109 \\
	USA}
\email{lenaji.math@gmail.com}

\begin{document}

\title{The Noether--Lefschetz theorem in arbitrary characteristic}

\thanks{During the preparation of this article, the author was supported by the National Science Foundation Graduate Research Fellowship Program under Grant No. DGE-1656466.}


\begin{abstract}
We show that if $X\subset\mathbb P^N_k$ is a normal variety of dimension $\geq 3$ and $H\subset\mathbb P^N_k$ a very general hypersurface of degree $d=4$ or $\geq 6$, then the restriction map $\Cl(X)\to\Cl(X\cap H)$ is an isomorphism up to torsion. If $\dim X\geq 4$, the result holds for $d\geq 2$. The proof uses the relative Jacobian of a curve fibration, together with a specialization argument, and the result holds over fields of arbitrary characteristic.
\end{abstract}

\maketitle

%\tableofcontents

\section{Introduction}

Given an ample divisor $Y$ on a variety $X$, one can ask whether the restriction map on divisors from $X$ to $Y$ is an isomorphism. In the case where the ambient variety is $\mathbb P^3$ over the complex numbers, the Noether--Lefschetz theorem shows that if $Y$ is a very general surface of degree $d\geq 4$ then the answer is yes \cite{noether82, lefschetz21}. Subsequent works of many authors, including \cite{cggh83,green84,ein1985,voisin88,voisin89}, have generalized this theorem in numerous directions, and for arbitrary smooth complex threefolds $X$ the Noether--Lefschetz theorem is known to hold for very general $Y\in|\mathcal L|$ with certain homological assumptions or if the line bundle $\mathcal L$ is sufficiently positive relative to $K_X$ \cite{moishezon67,voisin03,joshi95,brevik-nollet20}.

The main result of this article is a Noether--Lefschetz result for the Weil divisor class group $\Cl$ of normal varieties in arbitrary characteristic:
\begin{thm}[Theorems~\ref{prop:injectivity} and~\ref{prop:surjectivity}]\label{thm:main_theorem}
Let $k$ be an algebraically closed field, $X$ a normal projective variety of dimension $\geq 3$ over $k$, $\mathcal L_0$ a very ample line bundle on $X$, and $d\in\mathbb Z$. For $Y\in|\mathcal L_0^{\otimes d}|$ the restriction map
\[\xymatrix{
\Cl(X) \ar[r] & \Cl(Y)}
\]
\begin{enumerate}
\item\label{item:main_thm_1}
is injective for $d\geq 2$ and general $Y$,
\item\label{item:main_thm_2}
is surjective up to torsion for $d=4$ or $\geq 6$ and very general $Y$ if $\dim X=3$ and $k\neq\overline{\mathbb F}_p$, and
\item\label{item:main_thm_3}
is surjective up to $(\chara k)$-power torsion for $d\geq 2$ and very general $Y$ if $\dim X\geq 4$ and $k\neq\overline{\mathbb F}_p$. If moreover $k$ has infinite transcendence degree over its prime subfield, then surjectivity holds without torsion.
\end{enumerate}
\end{thm}
The omission of $d=5$ in~Theorem~\ref{thm:main_theorem}\eqref{item:main_thm_2} is an artifact of the method of proof (see the discussion before Proposition~\ref{prop:surj_for_2A+B}). Very general in Theorem~\ref{thm:main_theorem} means contained in a \emph{nonempty} subset that is the complement of a countable union of proper closed subvarieties of $|\mathcal L_0^{\otimes d}|$. The assumption that $k\neq\overline{\mathbb F}_p$ in part~\eqref{item:main_thm_2} is necessary: if $X=\mathbb P^3, \mathcal L_0=\mathcal O(1)$, and $d$ is even, then the Tate conjecture implies that every smooth $Y$ over $\overline{\mathbb F}_p$ has even Picard rank, so the complement of this union contains no $\overline{\mathbb F}_p$-points. In particular this happens for $d=4$ and $p\geq 5$ \cite{charles13}.

In characteristic $0$, a Noether--Lefschetz theorem for class groups of normal varieties was previously proven by Ravindra and Srinivas for very general $Y\in|\mathcal L|$ under the assumptions that $\mathcal L$ is basepoint free and ample line bundle, and that $K_X\otimes\mathcal L$ is globally generated \cite{rs09}. For singular varieties, using $\Cl$ rather than $\Pic$ gives the correct generalization, since the example of a cone over a smooth variety of higher Picard rank shows that surjectivity for Picard groups fails. Bruzzo--Grassi \cite{bruzzo-grassi12,bruzzo-grassi18} and Bruzzo--Grassi--Lopez \cite{bruzzo-grassi-lopez20} also studied the Noether--Lefschetz problem in characteristic $0$ for certain classes of mildly singular threefolds.

In arbitrary characteristic, for $d\geq 1$ and $Y$ \emph{very} general, injectivity of the restriction map on $\Cl$ was previously shown by Weil \cite{weil54}. Theorem~\ref{thm:main_theorem}\eqref{item:main_thm_1} improves injectivity to a \emph{general} member.
For surjectivity, when $\dim X=3$ the only previous result in positive characteristic is due to Deligne, who, using $\ell$-adic cohomology to extend Lefschetz's topological arguments, showed that the classical Noether--Lefschetz theorem for surfaces in $\mathbb P^3$ (and a generalization for complete intersections in $\mathbb P^N$, originally due to Lefschetz over $\mathbb C$) is true in positive characteristic \cite{sga7ii}. Away from this complete intersection case, Theorem~\ref{thm:main_theorem}\eqref{item:main_thm_2} in characteristic $p>0$ is new.

When $\dim X\geq 4$ one one can obtain stronger results, and the question is known as the Grothendieck--Lefschetz problem. In this higher-dimensional setting, Grothendieck proved that if $X$ is smooth and $Y$ is ample, in characteristic $0$ the restriction map $\Pic(X)\to\Pic(Y)$ is an isomorphism (and still injective if $\dim X=3$) \cite{sga2}. For normal ambient varieties in characteristic $0$, under the same dimension assumptions Ravindra and Srinivas proved the analogous statement for class groups for general $Y$ in a basepoint free and ample linear system \cite{rs06}, by generalizing Grothendieck's strategy on a resolution of singularities. Brevik and Nollet have shown a version incorporating a base locus \cite{brevik-nollet16}. In positive characteristic, the Grothendieck--Lefschetz theorem for Picard groups on smooth X still holds up to finite $p$-power torsion by replacing Kodaira vanishing with asymptotic Serre vanishing in the proof. For singular varieties, however, the failure of vanishing and Bertini theorems in positive characteristic pose serious obstacles to generalizing the argument of \cite{rs06} to positive characteristic.

To prove Theorem~\ref{thm:main_theorem}, we will study the divisors on $X$ using a family of abelian varieties---namely the relative Jacobian of a curve fibration. N\'eron used this approach in his thesis \cite{neron52}, where he related divisors on $X$ to the Jacobian of the generic fiber in his proof of the Theorem of the Base. This idea of fibering by curves dates back to Picard, and it has also been used to construct the Picard variety \cite{matsusaka-picard-1,igusa52,chow52}.

\subsection{Comparison with previously known results} Before giving an outline of the paper, we highlight which portions of Theorem~\ref{thm:main_theorem} are new.
\begin{enumerate}
\item Injectivity was previously known in characteristic $0$ \cite{rs06}, in arbitrary characteristic for \emph{very} general $Y$ \cite{weil54}, and in arbitrary characteristic modulo $\chara k$-power torsion when $X$ is smooth \cite{sga2}. Theorem~\ref{thm:main_theorem}\eqref{item:main_thm_1} is new in the case $X$ is singular, $\chara k=p>0$, and $Y$ is general.
\item When $\dim X=3$, surjectivity was previously known in characteristic $0$ for normal $X$ and very general $Y\in|\mathcal L|$ if $K_X\otimes\mathcal L$ is globally generated and $\mathcal L$ is basepoint free and ample \cite{rs09}. Theorem~\ref{thm:main_theorem}\eqref{item:main_thm_2} is the first surjectivity result without any assumptions on positivity of $\mathcal L$ relative to $K_X$.

In positive characteristic, Deligne showed surjectivity of $\Pic(\mathbb P^N)\to\Pic(Y)$ for $Y$ a very general complete intersection surface that is not a $(2,2)$ complete intersection in $\mathbb P^4$ or a degree $\leq 3$ surface in $\mathbb P^3$ \cite{sga7ii}. Aside from Deligne's result, Theorem~\ref{thm:main_theorem}\eqref{item:main_thm_2} is the first surjectivity result in characteristic $p>0$.

\item When $\dim X\geq 4$, surjectivity was known in characteristic $0$ \cite{rs06} and modulo $p$-power torsion for smooth $X$ when $\chara k =p>0$ \cite{sga2} (both of these results hold without the ``very general" requirement). To the author's knowledge, Theorem~\ref{thm:main_theorem}\eqref{item:main_thm_3} is the first Grothendieck--Lefschetz result for singular varieties in positive characteristic. (The very general assumption in Theorem~\ref{thm:main_theorem}\eqref{item:main_thm_3} is needed for the current method of proof, which uses Chow varieties.)

\end{enumerate}

\subsection{Outline}
The proof of Theorem~\ref{thm:main_theorem} uses two main ingredients: the relative Jacobian and a degeneration argument. The first step is to fiber $X$ by complete intersection curves in $|\mathcal L|$ and interpret divisors on $X$ using sections of the relative Jacobian of this curve fibration. For the subgroup $\Cl^0$ of algebraically trivial divisors, we use the theory of the Chow trace to show that $\Cl^0$ is preserved under restriction to a general $Y\in|\mathcal L|$. For injectivity of the restriction map on $\Cl$ (Theorem~\ref{prop:injectivity}), we use N\'eron's theorem on specialization of sections of abelian schemes to show that the restriction map on $\Cl$ to ``many" complete intersection curves is injective.

The more difficult part is surjectivity modulo torsion (Theorem~\ref{prop:surjectivity}). First, we assume the ground field $\mathbb K$ has infinite transcendence degree over its prime subfield. The first step is to use a result of Graber--Starr on extending sections of abelian schemes to show surjectivity of the restriction map to a very general \emph{reducible} complete intersection surface $S_1+S_2$. Then, we degenerate a very general $T\in|\mathcal L^{\otimes 2}|$ to $S_1+S_2$ to show that the restriction map $\Cl(X)\to\Cl(T)$ is surjective modulo torsion. This degeneration argument uses a conjecture of Koll\'ar, proven by Voisin, on independence of points in the support of a sufficiently ample divisor on a curve. Obtaining the result for odd multiples of $\mathcal L$ requires an additional degeneration argument, and this two-step degeneration for odd multiples of $\mathcal L$ is the reason $d=5$ is not covered in Theorem~\ref{thm:main_theorem}. Finally, to obtain the result for any algebraically closed field $k\neq\overline{\mathbb F}_p$, we use results of Andr\'e, Ambrosi, and Christensen on specializing N\'eron--Severi groups.

The outline of the paper is as follows. Section~\ref{section:preliminaries} contains relevant background and preliminaries, including Bertini-type theorems that impose the positivity assumptions in Theorem~\ref{thm:main_theorem}. Injectivity and the relative Jacobian are discussed in Section~\ref{sec:inj}. Surjectivity is shown in Section~\ref{sec:surj}.

\subsection*{Acknowledgements} The author is grateful to her advisor, J\'anos Koll\'ar, for his wisdom, patience, and generosity. She would also like to thank Claire Voisin for granting permission to use her proof of Proposition~\ref{lem:kol_sections_indep_conj}; Robin Hartshorne for helpful comments and questions, especially regarding \S\ref{sec:quartic_surfaces}; Bjorn Poonen for suggesting the extension to other fields in \S\ref{section:other_fields}; Joe Waldron for answering many questions and for comments on an early draft of this article; Scott Nollet for suggestions on the presentation; Remy van Dobben de Bruyn and Takumi Murayama for helpful conversations; and the referee for their helpful suggestions.

\section{Preliminaries}\label{section:preliminaries}

In this section we collect relevant background and preliminary results on class group of normal varieties (\S\ref{section:cl_background}), the notion of very general (\S\ref{sec:very_general}),  and on N\'eron's specialization theorem and thin sets (\S\ref{sec:n'eron}). In \S\ref{sec:bertini} we review some Bertini theorems and show some results about elements of sufficiently ample linear systems that we will need to use later in the paper.

\subsection{Class groups of normal varieties}\label{section:cl_background}

Let $X$ be a normal proper variety over an algebraically closed field $k$. When we say divisor without specifying further, we refer to Weil divisors. We use $\sim$ to denote linear equivalence of divisors, and if $D$ is a divisor on $X$, then by abuse of notation we will also write $D\in\Cl(X)$ for its linear equivalence class.

\begin{defn}[{\cite[Definition 15]{kol_mumford}}]
Let $\Cl(X)$ denote the group of divisors modulo linear equivalence, and let $\Cl^0(X)$ be the subgroup of divisors that are algebraically equivalent to 0. The \defi{N\'eron--Severi class group} of $X$ is the quotient $\Clns(X)=\Cl(X)/\Cl^0(X)$.
\end{defn}
The natural inclusions $\Pic^0(X)\subset\Cl^0(X)$ and $\NS(X)\subset\Clns(X)$
are both isomorphisms if and only if every Weil divisor is Cartier (for example if $X$ is locally factorial).

The class group can also be given a scheme structure via the Albanese variety.
Note that here, following \cite{kol_determines}, we use Albanese variety to refer to the classical, pre-Grothendieck notion of the Albanese variety \cite{weil1950} (see also \cite[128]{kol_determines}): the \defi{Albanese variety} of a variety $X$ over an algebraically closed field is the target of the universal \emph{rational map} $\mathrm{alb}_X\colon X\dashrightarrow\textbf{Alb}(X)$ from $X$ to an abelian variety.
\begin{thm}[{\cites{neron52}{lang-neron}[Theorem 17]{kol_mumford}[Section 14]{kol_determines}}]\label{thm:structure_class_group}
Let $X$ be a normal proper variety over an algebraically closed field $k$. Then
\[ \mathbf{Cl}^0(X) = \mathbf{Alb}(X)^*\]
where $\mathbf{Alb}(X)^*$ is the dual of the Albanese variety. The natural injection $\Cl^0(X)\to \mathbf{Cl}^0(X)(k)$ is an isomorphism, and $\Clns(X)$ is a finitely-generated abelian group.
\end{thm}
In the literature, $\mathbf{Cl}^0$ is sometimes called the Picard variety \cite{weil1950}. 
We refer the reader to \cite{kol_determines} for a modern treatment and to \cite{kleiman_fga_explained} for more on the history.

\begin{lem}[{\cite[Lemma 18]{kol_mumford}}]\label{lem:cl_birational}
Let $\beta\colon X'\to X$ be a proper birational morphism between normal varieties, with exceptional divisors $E_i$. Then $\beta_*$ induces isomorphisms
\begin{gather*}
\Cl^0(X')\cong\Cl^0(X), \\
\Clns(X')/\bigoplus_i\mathbb Z[E_i] \cong\Clns(X).
\end{gather*}
\end{lem}

Base changing to a larger algebraically closed field may enlarge $\Cl^0$ but will not affect $\Clns$.
\begin{lem}\label{lem:clns_field_extn}
Let $X$ be a normal variety over an algebraically closed field $k$. Then $\Clns(X_{\mathbb K})\cong\Clns(X)$ for any algebraically closed field $\mathbb K$ containing $k$.
\end{lem}

\begin{proof}
By \cite[124.3]{kol_determines} there is a normal proper variety $X'$ that is birational to $X$ and such that $\Cl^0(X')=\Pic^0(X')$, and $\NS(X'_{\mathbb K})\cong\NS(X')$ by \cite[Proposition 3.1]{mp12}.
\end{proof}

\subsection{Very general}\label{sec:very_general}

Let $X$ be a variety defined over an algebraically closed field $k$. A general choice will mean one made outside a finite union of proper closed $k$-subvarieties of the parameter space. When we use the term ``very general" without any further specification, we refer to a choice made in a nonempty subset given by the complement of a countable union of proper closed $k$-subvarieties. If $k$ is uncountable then the complement of such a countable union automatically contains a $k$-point.

In the latter section of this article, in order to make our very general choices more concrete, we will assume our algebraically closed field $\mathbb K$ has infinite transcendence degree, fix a field of definition, and make choices very general relative to this fixed field. In Weil's terminology $\mathbb K$ is a universal domain \cite[Chapter I, \S1]{weil-foundations}.
Although in practice the only countable fields this adds are the algebraic closures of $\mathbb Q(t_i \mid i\in\mathbb N)$ and $\mathbb F_p(t_i \mid i\in\mathbb N)$, it clarifies some instances when we need to make simultaneously very general choices.

\begin{defn}[{\cite[Chapter IV, \S1; Chapter IX, \S6]{weil-foundations}}]\label{defn:weil-fields}
Let $\mathbb K$ be a field and $X$ a variety over $\mathbb K$.
Then $X$ can be defined over a field that is finitely generated over the prime subfield, and when we say a \defi{field of definition} $k_0$ for $X$, we mean an algebraically closed field $k_0\subset\mathbb K$ over which $X$ is defined and such that $k_0$ has finite transcendence degree over the prime subfield ($\mathbb Q$ or $\mathbb F_p$). Note that $k_0$ is algebraic over a finitely-generated field and in particular is countable.

The Chow variety of $X$ is defined over $k_0$. A subvariety $Y$ of $X$ is given by a $\mathbb K$-point of $\mathrm{Chow}(X)$, and
$Y$ is said to be defined over a subfield $k_0\subset k'\subset\mathbb K$ if it is of the form $Y=Y'\otimes_{k'} \mathbb K$ for a subscheme $Y'$ of $X\otimes_{k_0} k'$. There is a unique smallest field of definition $k^Y$ of $Y$ in $X$, which is the residue field of the morphism to $\Spec\mathbb K\to\mathrm{Chow}(X)$ giving $Y$ \cite[3.18.1 and Definition 3.21]{kol_moduli_book}. 
Note that $k^Y$ is finitely generated and is generally not algebraically closed.
If the coordinates of the $\mathbb K$-point $[Y]$ do not satisfy any algebraic equation over $k_0$ (or equivalently if the transcendence degree of $k^Y$ over $k_0$ is equal to the dimension of the component of the Chow variety containing $[Y]$), we will say that $Y$ is \defi{very general over $k_0$}.
If $\mathbb K$ is assumed to have infinite transcendence degree over the prime subfield then we can ensure that such $Y$ exist.

We will say that $Y_1,\ldots,Y_r$ are \defi{very general over $k_0$} if each $Y_i$ is very general over $k_0$ and the compositum of the fields $k^{Y_j}$ for $j\neq i$. In Weil's language, the fields $k^{Y_1},\ldots,k^{Y_r}$ are ``independent" or ``free with respect to each other" \cite[Chapter I, \S2]{weil-foundations}.
\end{defn}

A very general choice over $k_0$ avoids any finite union of proper $k_0$-subvarieties.

Let $X$ be a variety over a field $\mathbb K$ of infinite transcendence degree, and let $\mathcal Y\to B$ be a family of complete intersections in $X$. For $b\in B(\mathbb K)$ let $Y_b$ be the corresponding fiber. If we show that the restriction map $\Cl(X)\to\Cl(Y_b)$ is surjective up to torsion whenever $b$ is very general over a field of definition for $X$, then we will have shown that the restriction map to the class group of the geometric generic member $\mathcal Y_{\overline{\eta}}$ of the family is surjective up to torsion. Over an uncountable field, this is equivalent to the result for very general fibers, i.e. outside of a countable union of proper subvarieties of $B$, see \cite[Section 3]{rs09}.

\subsection{Thin sets and N\'eron's theorem}\label{sec:n'eron}

Here we review some relevant facts about N\'eron's specialization theorem that we will use in Section~\ref{sec:inj}.

\begin{defn}[{\cites[Chapter 3]{serre92}[Chapter 12]{fm08}}]\label{defn:thin_set}
Let $k_1$ be a field that is finitely generated over its prime subfield.
A subset $T\subset\mathbb P^m_{k_1}(k_1)$ is \defi{thin} if $T=f(X(k_1))$ for some $k_1$-variety $X$ and morphism $f\colon X\to\mathbb P^m_{k_1}$ that is separable and generically finite onto its image and that admits no rational section.

Every thin subset of $\mathbb P^m_{k_1}(k_1)$ is contained in a finite union of the following two types of thin sets:\begin{enumerate}\item Thin subsets contained in proper subvarieties, and \item $f(X(k_1))$ for some separable dominant morphism $f\colon X\to\mathbb P^m_{k_1}$ with $\dim X=m$ and $\deg(f)\geq 2.$\end{enumerate}
If $k_1$ is a finitely-generated field not contained in $\overline{\mathbb F}_p$, then $\mathbb P^m_{k_1}(k_1)$ is not thin for $m\geq 1$ \cites[Theorem 3.4.1]{serre92}[Theorem 12.10]{fm08}.
Following \cite{kol_determines}, for an arbitrary field $k$ we say that $T\subset\mathbb P^m_k(k)$ is \defi{field-locally thin} if $T\cap\mathbb P^m(k_1)$ is thin for every finitely-generated subfield $k_1\subset k$.
\end{defn}

\begin{thm}[{\cite[Theorem 6]{neron52}}]\label{thm:neron's_theorem_av}
Let $k$ be a field, $U\subset\mathbb P^m_k$ an open subset, and $\mathcal A\to U$ an abelian scheme.
Assume the Mordell--Weil group $\mathcal A_\eta(K)$ is finitely generated, where $K$ is the function field of $\mathbb P^m_k$ and $\mathcal A_\eta$ is the generic fiber of $\mathcal A$.

Then there exists a subset $N\subset U(k)$ containing the complement of a field-locally thin set and such that the specialization map \[\xymatrix{ \mathcal A_\eta(K) \ar[r] & \mathcal A_b(k)}\]
is injective for $b\in N$.
\end{thm}

\subsection{Bertini theorems}\label{sec:bertini}

In this section we recall some Bertini theorems for very ample divisors in arbitrary characteristic. We also show that if a line bundle is sufficiently positive, then the reducible locus is small (Lemma~\ref{lem:reducible_in_codim_2}) and the sectional genus is large (Lemma~\ref{lem:cokernel_to_Jac(C)}).

\begin{thm}[{\cite[Theorem 3.4.10]{fov99}}]\label{thm:bertini_irreducibility}
Let $X$ be a variety over an infinite field $k$, $D$ a Cartier divisor on $X$, and $|V|\subset|D|$ a finite-dimensional linear system. If \begin{enumerate}\item $|V|$ is not composed with a pencil (meaning the image of $X$ under the morphism defined by $|V|$ has dimension $\geq 2$), and \item $\codim_X\mathrm{Bs}(|V|)\geq 2$,\end{enumerate}then a general member of $|V|$ is irreducible.
\end{thm}

\begin{thm}[{\cite[Theorem 3.4.14]{fov99}}]\label{lem:bertini_va}
Let $X\subset\mathbb P^N_k$ be a projective scheme over an infinite field $k$. If $X$ is regular (resp. normal, reduced, regular in codimension $c$), then for a general hyperplane $H\subset\mathbb P^N_k$, the intersection $H\cap X$ has the same property.
\end{thm}

\begin{thm}[{\cite[Corollary 3.4]{gk19}}]\label{lem:bertini_normal}
Let $k$ be an infinite field, $X$ a normal equidimensional quasi-projective $k$-scheme of dimension $\geq 1$, and $X\subset\mathbb P^N_k$ a locally closed embedding. Then a general hypersurface $H\subset\mathbb P^N_k$ satisfies \begin{enumerate}\item $X\cap H$ is normal, and \item $X^{\mathrm{reg}}\cap H$ is regular.\end{enumerate}
\end{thm}

We will need to use Grothendieck's connectedness lemma in the proof of Lemma~\ref{lem:reducible_in_codim_2} below.
\begin{lem}[{\cite[Expos\'e XIII, Theorem 2.1]{sga2}}
]\label{lem:loc.conn.1.lem}
Let $(x,X)$ be a local, excellent, $S_2$ scheme of pure dimension $\geq 3$ and $x\in D\subset X$ a Cartier divisor. Then $D\setminus\{x\}$ is connected.
\end{lem}

In the proof of Theorem~\ref{thm:main_theorem}, we want to work with linear systems where the locus of reducible members has codimension $\geq 2$. This can sometimes fail, as in the case of $|\mathcal O_{\mathbb P^2}(2)|$ where the closure of the reducible plane conics is a divisor, but Lemma~\ref{lem:reducible_in_codim_2} will ensure these examples doesn't happen if the linear system is sufficiently positive.

\begin{lem}\label{lem:add_to_bir_lin_sys}
Let $X$ be a variety, $\mathcal L$ a line bundle and $\Gamma\subset H^0(X,\mathcal L)$ a non-zero subspace, and $\mathcal B$ a line bundle defining a birational map. Then the linear system $|\Gamma|+|\mathcal B|=\mathbb P \{r\otimes s \mid r\in \Gamma, s \in H^0(X,\mathcal B)\}\subset |\mathcal L\otimes\mathcal B|$ defines a birational map.
\end{lem}

\begin{proof}Let $s_0,\ldots,s_N$ be a basis of $H^0(X,\mathcal B)$ and $r_0,\ldots,r_m$ a basis of $\Gamma$.
Then $\{r_i\otimes s_j \mid 0\leq i\leq m, 0\leq j\leq N\}$ is a basis of $|\Gamma|+|\mathcal B|$, so the rational map $\phi\colon X\dashrightarrow \mathbb P^{(m+1)(N+1)}$ it defines is given (up to isomorphism) by \[ x\mapsto [r_0s_0(x):r_0s_1(x):\cdots:r_0s_N(x):r_1s_0(x):\cdots:r_ms_N(x)]. \] Denote the coordinates of $\mathbb P^{(m+1)(N+1)}$ by $[y_{00}:y_{01}:\cdots:y_{mN}]$. For each $i$, define $\pi_i\colon\mathbb P^{(m+1)(N+1)}\dashrightarrow\mathbb P^N$ to be the projection onto the coordinates $[y_{i0}:y_{i1}:\cdots:y_{iN}]$.
Then $\pi_i\circ\phi$ agrees with the rational map defined by $|\mathcal B|$ (away from $(s_i=0)$) so its image has dimension $n$. So $\dim\phi(X)= n$. \end{proof}

\begin{lem}\label{lem:reducible_in_codim_2}
Let $X$ be a projective variety of dimension $n\geq 2$ over an algebraically closed field $k$, and assume either
\begin{enumerate}[label=(\alph*)]
\item
$X$ is normal, or
\item
$\dim\Sing(X)=0$.
\end{enumerate}
Let $\mathcal L_0$ be a very ample line bundle on $X$.
Then the set of reducible divisors in $|\mathcal L_0^{\otimes d}|$ is contained in a closed subscheme of $|\mathcal L_0^{\otimes d}|$ of codimension $\geq 2$ for
\begin{enumerate}
\item
$d\geq 2$ if the pair $(X,\mathcal L_0)$ is not $(\mathbb P^2,\mathcal O_{\mathbb P^2}(1))$, and
\item
$d\geq 3$ in general.
\end{enumerate}
\end{lem}

\begin{proof}
The case where $X=\mathbb P^n$ and $\mathcal L_0=\mathcal O_{\mathbb P^n}(1)$ follows from an explicit computation, so we assume that $(X,\mathcal L_0)\neq(\mathbb P^n,\mathcal O_{\mathbb P^n}(1))$.
For each smooth closed point $x\in X$ and line bundle $\mathcal M$ on $X$, define $|B_x^\mathcal M|=\{D\in|\mathcal M| \mid x\in\Sing D\}$.

First we show that a general member of the linear system $|B^{\mathcal L_0^{\otimes d}}_x|$ is irreducible.
Since $(X,\mathcal L_0)\neq(\mathbb P^n,\mathcal O(1))$, the set
$|B_x^{\mathcal L_0}|=\{H\cap X \mid H\in|\mathcal O_{\mathbb P^{h^0(X,\mathcal L_0)-1}}(1)| \text{ and } T_{\phi_{|\mathcal L_0|}(X),x} \subset H \}$ is nonempty, so by Lemma~\ref{lem:add_to_bir_lin_sys} the linear system $|B^{\mathcal L_0}_x|+|\mathcal L_0^{\otimes(d-1)}|$ defines a birational map to its image.
Since $|B^{\mathcal L_0}_x|+|\mathcal L_0^{\otimes(d-1)}|\subset| B^{\mathcal L_0^{\otimes d}}_x|$, the span of these reducible divisors is also contained in $|B^{\mathcal L_0^{\otimes d}}_x|$ and so the morphism $|B^{\mathcal L_0^{\otimes d}}_x|$ defines is birational to its image. In particular its image has dimension $n\geq 2$.
Since the base locus of $| B^{\mathcal L_0^{\otimes d}}_x|$ is $x$, which has codimension $n\geq 2$ in $X$, a general member of this linear system is irreducible by Theorem~\ref{thm:bertini_irreducibility}.

Replacing $X$ by its image under the closed embedding defined by $\mathcal L_0^{\otimes d}$, we may assume that $\mathcal L_0^{\otimes d}=\mathcal O_{\mathbb P^N}(1)|_X$, where $\mathbb P^N=|\mathcal L_0^{\otimes d}|$, and that $X$ is not contained in a hyperplane. Let $\mathcal C=\overline{\{(x,H) \mid x\in X^{\mathrm{reg}}\text{ and } T_{X,x}\subset H\}}\subset X\times(\mathbb P^N)^*$ be the conormal variety of $X$, with projections $\pi_1\colon X\times(\mathbb P^N)^*\to X$ and $\pi_2\colon X\times(\mathbb P^N)^*\to(\mathbb P^N)^*$.
The dual variety $\pi_2(\mathcal C)$ is irreducible since $\mathcal C$ is irreducible of dimension $N-1$.

Now consider the locus $S_{\text{reducible}}$ of hyperplanes whose intersection with $X$ is reducible.
This is a constructible subset of the closed subvariety of $|\mathcal L_0^{\otimes d}|$ parametrizing hyperplanes not regular in codimension 0 \cite[Theorem 3.3.14]{fov99}, and $S_{\text{reducible}}\subsetneq|\mathcal L_0^{\otimes d}|$ by Theorem~\ref{thm:bertini_irreducibility}.
Note that by the Enriques--Severi--Zariski lemma, if $H\cap X$ is reducible then it is singular.
So if $X$ is smooth then $\overline{S_{\text{reducible}}}\subset \pi_2(\mathcal C)$, and this containment is proper by the above discussion, so $\overline{S_{\text{reducible}}}$ has codimension $\geq 2$ in $|\mathcal L_0^{\otimes d}|$.

If $X$ is not smooth, suppose first that $\Sing(X)$ has dimension 0.
For a closed point $x\in\Sing(X)$ consider the linear system $|{\mathcal L_0}_x|$ of divisors passing through $x$.
By Lemma~\ref{lem:add_to_bir_lin_sys} $|{\mathcal L_0}_x|+|\mathcal L_0^{\otimes(d-1)}|$ is not composed with a pencil, and hence neither is $|{\mathcal L_0^{\otimes d}}_x|$.  By Theorem~\ref{thm:bertini_irreducibility} a general member of $|{\mathcal L_0^{\otimes d}}_x|$ is irreducible, and so the reducible members $R_x$ of $|{\mathcal L_0^{\otimes d}}_x|$ form a closed subset of $\pi_1^{-1}(x)$ of dimension $\leq N-2$.

Let $H\in S_{\text{reducible}}$. If the intersection of two components of $H\cap X$ contains a smooth point of $X$, then $H\in \pi_2(\mathcal C)$. If not, then let $x\in\Sing(X)$ be a closed point in the intersection of two components of $H\cap X$. Then $H\in \pi_2(R_x)$. Since $\Sing(X)$ is a finite set, $\bigcup_{x\in\Sing(X)} B_x$ has dimension at most $N-2$, so its image in $(\mathbb P^N)^*$ has dimension at most $N-2$. Thus, in the case where $\dim\Sing(X)=0$ we have shown that $\codim_{|\mathcal L_0^{\otimes d}|}(\overline{S_{\text{reducible}}})\geq 2$.

Now assume $X$ is normal.
If $n=2$ then $\dim\Sing(X)\leq 0$, so we may assume that
$n\geq 3$. Then $S_{\text{reducible}}$ is the union of the two sets
\begin{enumerate}
\item
$\{H\mid H\cap X \text{ is reducible and }H\supset T_{X,x} \text{ for some closed }x\in X^{\mathrm{reg}}\}$ and
\item
$\{H \mid H\cap X \text{ is reducible and }\Sing(H\cap X)\subset\Sing(X)\}$. \end{enumerate}
The closure of the first set has codimension $\geq 2$ in $|\mathcal L_0^{\otimes d}|$ by the argument in the smooth case, so it remains to consider the second set.

For such an $H$, denote the irreducible components of $H\cap X$ by $Z_1,\ldots,Z_e$. Each intersection $Z_i\cap Z_j\subset\Sing(H\cap X)$ for $i\neq j$. For each codimension 3 point $x\in X$ contained in $Z_i\cap Z_j$, $\Spec\mathcal O_{X,x}\setminus\{x\}$ is connected by Lemma~\ref{lem:loc.conn.1.lem}, so $Z_i\cap Z_j$ has codimension 2 in $X$. Since $X$ is regular in codimension 1, and $Z_i\cap Z_j\subset\Sing(X)$ both have codimension 2 in $X$, we have that $Z_i\cap Z_j$ is a component of $\Sing(X)$. So $H\cap X$ must contain an $(n-2)$-dimensional irreducible component of $\Sing(X)$. By assumption $n-2\geq 1$, so requiring that $H$ contains a irreducible component of $\Sing(X)$ imposes a codimension $\geq 2$ condition on the elements of $|\mathcal L_0^{\otimes d}|$. So again in this case $\overline{S_{\text{reducible}}}$ has codimension $\geq 2$ has codimension $\geq 2$ in $|\mathcal L_0^{\otimes d}|$.
\end{proof}

\begin{rem}
If $n\geq 3$ then Lemma~\ref{lem:reducible_in_codim_2} holds for $d=1$ as well without the normality assumption \cites[Lemme 3]{neron-samuel52}[Lemme 4]{weil54}.
For surfaces in characteristic $0$, the Kronecker--Castelnuovo theorem \cite{castelnuovo1894} says that the general tangent hyperplane section of $X$ is reducible, then $X$ is either a ruled surface or the Veronese embedding of $\mathbb P^2$ in $\mathbb P^5$. Castelnuovo's proof uses differential geometry, and the author is not aware of a positive characteristic analogue of his result.
\end{rem}

\begin{cor}\label{cor:curves_reducible_in_codim_2}
Let $X\neq\mathbb P^2$ be a projective variety of dimension $n\geq 2$ over an algebraically closed field $k$, and assume that either
\begin{enumerate}[label=(\alph*)]
\item
$X$ is normal, or
\item
$\dim\Sing(X)=0$.
\end{enumerate}
Fix $r\leq n-1$, $\mathcal L_0$ a very ample line bundle on $X$, and $d\geq 2$ an integer.
Then the locus of reducible complete intersections of $r$ members of $|\mathcal L_0^{\otimes d}|$, is contained in a closed subset of $|\mathcal L_0^{\otimes d}|$ of codimension $\geq 2$.
\end{cor}

\begin{proof}
The result is clear if the pair $(X,\mathcal L_0^{\otimes d})$ is $(\mathbb P^n,\mathcal O(1))$, so we assume otherwise.
Induct on $r$:
the $r=1$ case is Lemma~\ref{lem:reducible_in_codim_2}.
If $r\geq 2$, let $\pi_1\colon |\mathcal L_0^{\otimes d}|\times\cdots\times|\mathcal L_0^{\otimes d}|\to |\mathcal L_0^{\otimes d}|$ denote the first projection from the $r$-fold product, and let $R\subset|\mathcal L_0^{\otimes d}|\times\cdots\times|\mathcal L_0^{\otimes d}|$ denote the set of $r$-tuples $(H_1,\ldots,H_r)$ such that $H_1\cap\cdots\cap H_r$ is reducible.
First assume $X$ is normal.
For each fixed $H_1\in|\mathcal L_0^{\otimes d}|$, we have that $\codim_{\pi_1^{-1}(H_1)}(R\cap\pi_1^{-1}(H_1))\geq 2$ if $H_1$ is normal by hypothesis, and $\codim_{\pi_1^{-1}(H_1)}(R\cap\pi_1^{-1}(H_1))\geq 1$ if $H_1$ is not normal. In the latter case, the locus of non-normal $H_1$ is contained in a codimension 1 subset of $|\mathcal L_0^{\otimes d}|$ by Bertini's theorem (Theorem~\ref{lem:bertini_normal}). So by dimension counting we have that $\codim_{|\mathcal L_0^{\otimes d}|\times\cdots\times|\mathcal L_0^{\otimes d}|}(R) \geq 2$. The $\dim\Sing(X)=0$ case follows by the same argument.
\end{proof}

We will need the following result on sectional genus to apply Proposition~\ref{lem:kol_sections_indep_conj} in the proofs of Propositions~\ref{prop:surj_for_L^2_to_surfaces} and \ref{prop:surj_for_2A+B}.

\begin{lem}\label{lem:cokernel_to_Jac(C)}
Let $X$ be a normal projective variety of dimension $n\geq 2$ over an algebraically closed field $k$ and $\mathcal L_0$ an ample and basepoint-free line bundle on $X$. For any smooth complete intersection curve $D$ of members of $|\mathcal L_0^{\otimes d}|$
such that $D$ is contained in the smooth locus of $X$, the cokernel of the restriction map
\[\xymatrix{
\mathbf{Cl}^0(X) \ar[r] & \mathbf{Jac}(D)}
\]
has positive dimension if
\begin{enumerate}
\item
$d\geq 2$ if the pair $(X,\mathcal L_0)$ is not $(\mathbb P^2,\mathcal O_{\mathbb P^2}(1))$, and
\item
$d\geq 3$ in general.
\end{enumerate}
\end{lem}

\begin{proof}
Let $C\subset X^{\mathrm{reg}}$ be a smooth complete intersection curve of members of $|\mathcal L_0|$. Since the morphism $\mathbf{Jac}(C)\to\mathbf{Alb}(X)$ is surjective \cite[Lemma 130]{kol_determines} we have $p_a(C)\geq\dim\mathbf{Alb}(X)=\dim\mathbf{Cl}^0(X)$. So we need to show that $p_a(D)>p_a(C)$ for $D$ in the statement of the lemma.

First suppose $X$ is a surface. The arithmetic genus of any member of $|\mathcal L_0^{\otimes d}|$ is equal to that of $\sum_{i=1}^d C_i$ for curves $C_i\in|\mathcal L_0|$, which is equal to
\[ p_a\left(\sum_{i=1}^d C_i\right) = \sum_{i=1}^d p_a(C_i) + (1-d) + \frac{d(d-1)}{2}C_1^2.\] If $d=2$ then $p_a(C_1+C_2)= p_a(C_1)+p_a(C_2)+C_1^2-1$, and this is strictly larger than $p_a(C_i)$ unless $p_a(C_i)=0$ and $C_1^2=1$.
In this case, $X=\mathbb P^2$ and $\mathcal L_0=\mathcal O_{\mathbb P^2}(1)$, and any plane curve of degree $d\geq 3$ has positive genus.

Now assume $n\geq 3$ and let $H_0\in|\mathcal L_0|$. Let $D$ be a complete intersection curve of members of $|\mathcal L_0^{\otimes d}|$, and let $C$ be a smooth complete intersection curve of members of $|\mathcal L_0|$. By adjunction,
\begin{equation*}\begin{split}
p_a(C) &= 1+\frac{H_0^{n-1}\cdot K_X+(n-1) H_0^n}{2}, \\
p_a(D) &= 1+\frac{d^{n-1}H_0^{n-1}\cdot K_X+d^n (n-1) H_0^n}{2}.
\end{split}\end{equation*}
We have $p_a(D) - p_a(C) \geq p_a(D) - d^{n-1} p_a(C)$ since $p_a(C)\geq 0$, and
\begin{equation}\label{eqn:genus-inequality}p_a(D) - d^{n-1} p_a(C) = 1-d^{n-1}+\frac{(d^n-d^{n-1}) (n-1)}{2} H_0^n. \end{equation}
Since $n\geq 3$ and since $\mathcal L_0$ is ample, the difference~\eqref{eqn:genus-inequality} is at least $d^n-2d^{n-1}+1$, which is positive for $d\geq 2$.

\end{proof}

In order to ensure that the conclusions of Corollary~\ref{cor:curves_reducible_in_codim_2} and Lemma~\ref{lem:cokernel_to_Jac(C)} hold, throughout this article we will assume $\mathcal L=\mathcal L_0^{\otimes d}$ for some very ample line bundle $\mathcal L_0$ on $X$ and integer $d\geq 2$.

\section{Injectivity}\label{sec:inj}

In this section we prove injectivity of the restriction map on class groups.

\begin{thm}[Corollary~\ref{cor:inj_to_surfaces}]\label{prop:injectivity}
Let $X$ be a normal projective variety of dimension $n\geq 3$ over an algebraically closed field $k$, let $\mathcal L_0$ be a very ample line bundle, and let $d\geq 2$ be an integer. Then the restriction map \[\xymatrix{ \Cl(X) \ar[r] & \Cl(Y)} \] is injective for a general complete intersection $Y$ of at most $n-2$ members of $|\mathcal L_0^{\otimes d}|$.
\end{thm}

Weil proved injectivity of the restriction map on $\Cl$ when $Y$ is very general over $k$ for $n\geq 3$ and $d\geq 1$ (and more generally for $X$ and $\mathcal L_0^{\otimes d}$ such that the reducible locus in $|\mathcal L_0^{\otimes d}|$ has codimension $\geq 2$) by considering a pencil in $|\mathcal L_0^{\otimes d}|$ \cite[Th\'eor\`eme 2]{weil54}. We will instead fiber $X$ by curves and use N\'eron's specialization theorem to show the result for \emph{general} members.

We describe this curve fibration and the relation between $\Cl(X)$ and sections of the relative Jacobian in \S\ref{sec:sections_of_rel_jac}. In \S\ref{sec:cl^0_iso} we show on the subgroup $\Cl^0\subset\Cl$ of algebraically trivial divisors that the restriction map $\Cl^0(X)\to\Cl^0(Y)$ is an isomorphism (Theorem~\ref{prop:cl^0_iso}). Then, in \S\ref{sec:injectivity_pf}, we prove Theorem~\ref{prop:injectivity}.

\subsection{Sections of the relative Jacobian}\label{sec:sections_of_rel_jac}

We will use the relative Jacobian to understand divisors. If $f\colon X\to B$ is a projective and flat morphism of relative dimension one, locally of finite presentation, and with geometrically reduced and irreducible fibers, then $\mathbf{Pic}_{X/B}$ is representable as a smooth separated $B$-scheme. If $f$ is smooth over $U\subset B$, then $\mathbf{Pic}^0_{X|_U/U}$ is an abelian $U$-scheme \cite[Theorem 9.3.1, Proposition 9.4.4]{blr}.

We now fix some notation and assumptions that will be used throughout the paper and describe the setting in which we will consider the relative Jacobian. Since the restriction map from $\Pic(\mathbb P^2)$ to any curve is injective, we may assume $X\neq\mathbb P^2$.

\begin{assumption}\label{assumption:inj}
Let $X\neq\mathbb P^2$ be a normal projective variety of dimension $n\geq 2$ over an algebraically closed field $k$ of arbitrary characteristic, $\mathcal L$ a very ample line bundle on $X$, and $|V|\subset|\mathcal L|$ a general linear system of dimension $n-1$. Assume that $\mathcal L=\mathcal L_0^{\otimes d}$ for an integer $d\geq 2$ and a very ample line bundle $\mathcal L_0$ on $X$. Fix a base point $x$ of the linear system $|V|$. We also fix a field of definition $k_0\subset k$ for $X$ and $|V|$, recalling by our conventions in Definition~\ref{defn:weil-fields} that this means $k_0$ is algebraically closed and has finite transcendence degree over the prime subfield.
\end{assumption}

The linear system $|V|$ defines a rational map $\phi\colon X\dashrightarrow\mathbb P^{n-1}$, which is a morphism away from the finite set $\Bs|V|$. Let $\phi'\colon X'\to\mathbb P^{n-1}$ be the normalization of the closure of the graph of $\phi$;
by Theorem~\ref{lem:bertini_normal} this is a smooth morphism over a dense open subset $U_{\mathrm{sm}}\subset\mathbb P^{n-1}\setminus\phi'(\Sing(X'))\subset\mathbb P^{n-1}$. Let $\beta\colon X'\to X$ be the induced birational morphism. This is an isomorphism away from the base locus of $|V|$, which by generality is contained in the regular locus of $X$, so the pullback $\beta^*\colon\Cl(X)\to\Cl(X')$ is defined. Let $E_x$ be the exceptional divisor over $x$.

For a subvariety $Z\subset X$ we denote by $Z'\subset X'$ its strict transform.
For a point $b\in\mathbb P^{n-1}$ the closure $C_b=\overline{\phi^{-1}(b)}$ of the fiber of $\phi$ is a complete intersection curve of members of $|V|$, and $C'_b=\phi'^{-1}(b)$. The conditions of Lemma~\ref{lem:reducible_in_codim_2} are satisfied, so the locus $R$ of reducible fibers has codimension $\geq 2$ in $|V|$. Let $U_{\mathrm{i}}$ denote the complement of the closure of $R$, let $U_{\mathrm{sm}}$ denote the locus of smooth fibers, and let $C'_\eta$ denote the generic fiber of $\phi'$, which is a curve over the function field $K$ of $\mathbb P^{n-1}$.

We consider the relative Jacobian of $\phi'\colon X'\to\mathbb P^{n-1}$ in the above setting.
Then
\begin{itemize}
\item
$\mathbf{Pic}^0_{X'|_{U_{\mathrm{i}}}/ U_{\mathrm{i}}}$ is a variety and smooth over $U_{\mathrm{i}}$, and
\item
$\mathbf{Pic}^0_{X'|_{U_{\mathrm{sm}}}/ U_{\mathrm{sm}}} \cong \mathbf{Pic}^0_{X'|_{U_{\mathrm{i}}}/ U_{\mathrm{i}}}\times_{U_{\mathrm{i}}} U_{\mathrm{sm}}$ with is an abelian $U_{\mathrm{sm}}$-scheme with zero section $b\mapsto\mathcal O_{C'_b}$. We frequently denote it by $\mathbf{Jac}(X'/\mathbb P^{n-1})$.
\end{itemize}

The section of $\phi'$ given by the exceptional divisor $E_x$ defines a rational map $X'\dashrightarrow\mathbf{Jac}(X'/\mathbb P^{n-1})$ that is a morphism on $U_{\mathrm{sm}}$, and for any $U\subset U_{\mathrm{sm}}$ we have an isomorphism
\[ \Pic^0(X'|_{U}) \cong\mathbf{Pic}^0_{X'|_{U_{\mathrm{i}}}/U_{\mathrm{i}}}(U)
\]
between equivalence classes of algebraically trivial line bundles and sections \cite[Proposition 8.1.4]{blr}.

We will study divisors on $X'$ by relating them to rational sections of the relative Jacobian \cite[Chapitre I, \S8]{neron52}.
An algebraically trivial Cartier divisor on $X'$ defines a section of $\mathbf{Pic}^0_{X'|_{U_{\mathrm{i}}}/U_{\mathrm{i}}}\to U_{\mathrm{i}}$ by pullback.
An algebraically trivial Weil divisor $D$ on $X'$ defines a \emph{rational} section $\sigma_D$ of $\mathbf{Pic}^0_{X'|_{U_{\mathrm{i}}}/U_{\mathrm{i}}}\to U_{\mathrm{i}}$ over $U_\mathrm{sm}$ by restricting first to the regular locus of $X'$ and then pulling back the Cartier divisor $\mathcal O_{X'^{\mathrm{reg}}}(D|_{X'^{\mathrm{reg}}})$. Note that this rational map does not in general extend over all of $U_{\mathrm{i}}$.

Twisting down by an appropriate multiple of $E_x$, we can extend this to any Weil divisor on $X'$:
\begin{lem}\label{lem:cl^0_to_sections}
The map \begin{equation}\label{eqn:cl(X')_to_sections}\xymatrixrowsep{0pc}\xymatrix{
\Cl(X')\ar[r] & \mathbf{Jac}(C'_\eta)(K)=\{\text{sections of }\mathbf{Jac}(X'/\mathbb P^{n-1}) \to U_\mathrm{sm}\}\\
D' \ar@{|->}[r] & (D'-\deg(D'|_{C'_\eta})E_x)|_{C'_\eta}}\end{equation}
is a surjection with kernel generated by $E_x$ and $\phi'^*\mathcal O_{\mathbb P^{n-1}}(a)$ for some $a\in\mathbb Z$.
\end{lem}

\begin{proof}
The morphism $\phi'$ admits a section, so a section $\sigma$ of $\mathbf{Jac}(X'/\mathbb P^{n-1})\to U_{\mathrm{sm}}$ is the same as a divisor class in $\Pic^0(X'|_{U_\mathrm{sm}})$ whose restriction to $C'_b$ agrees with $\sigma(b)\in\Jac(C'_b)$. So surjectivity follows from surjectivity of $\Cl(X')\to\Cl(X'|_{U_{\mathrm{sm}}})\cong\Pic(C'_\eta)$.

Next we consider the kernel. The codimension 1 irreducible components $\{\gamma_i\}_i$ of $\mathbb P^{n-1}\setminus U_{\mathrm{sm}}$ are degree $a_i$ hypersurfaces in $\mathbb P^{n-1}$. By Lemma~\ref{cor:curves_reducible_in_codim_2}, each preimage $\phi'^{-1}(\gamma_i)$ is irreducible, so the kernel is generated by $\phi'^*\mathcal O_{\mathbb P^{n-1}}(a)$ where $a=\gcd\{a_i\}$.
So for a divisor $D'$ to be in the kernel of the map (\ref{eqn:cl(X')_to_sections}), it must satisfy $D'-\deg(D'|_{C'_\eta})E_x \in \langle\phi'^*\mathcal O_{\mathbb P^{n-1}}(a)\rangle$. Since $d\geq 2$ in Assumption~\ref{assumption:inj} we have $\#\Bs|V|\geq 2$, so $E_x$ is not the only exceptional divisor of $\beta$, and for any other  exceptional divisor $E'$ the restrictions $E_x|_{C'_\eta}$ and $E'|_{C'_\eta}$ are independent in $\Pic(C'_\eta)$. So the kernel of (\ref{eqn:cl(X')_to_sections}) is generated by $E_x$ and $\phi'^*\mathcal O_{\mathbb P^{n-1}}(a)$.
\end{proof}

The image of $\Cl(X)$ is disjoint from the subgroup of $\Cl(X')$ generated by $E_x$ and $\phi'^*\mathcal O_{\mathbb P^{n-1}}(a)$, so in particular we have:
\begin{cor}\label{cor:cl(x)_to_jac_generic}
The map $\Cl(X)\to\Jac(C'_\eta)$ induced by the homomorphism (\ref{eqn:cl(X')_to_sections}) of Lemma~\ref{lem:cl^0_to_sections} is injective.
\end{cor}

\begin{rem}\label{rem:cl^0_trace}
For an abelian scheme $\mathcal A\to B\to\Spec k$ we denote by $\Tr_{B/k}(\mathcal A)$ and $\Im_{B/k}(\mathcal A)$ the Chow $B/k$-trace and Chow $B/k$-image, respectively. These are abelian $k$-varieties that depend only on the field extension $K/k$ and the abelian variety $\mathcal A\times_B\Spec K$ over $K:=k(B)$ \cite[Remark 3.3]{graber-starr13}. We refer the reader to \cite{conrad06} and \cite[Section 3]{graber-starr13} for a modern account.

Chow introduced the $K/k$-trace and $K/k$-image and used this theory to construct $\Cl^0$ and the Albanese variety \cite{chow52,chow55}. 
Specifically, there is a canonical isomorphism
\[\mathbf{Cl}^0(X')\cong\Tr_{U_{\mathrm{sm}}/k}(\mathbf{Jac}(X'/\mathbb P^{n-1})).\]
In one direction, from the rational maps
\[\xymatrix{
X' \ar@{-->}[r] & \mathbf{Jac}(X'/\mathbb P^{n-1}) \ar[r]^-\cong & \mathbf{Jac}(X'/\mathbb P^{n-1})^\vee \ar[r] & \Im_{U_{\mathrm{sm}}/k}(\mathbf{Jac}(X'/\mathbb P^{n-1})^\vee)\times U_{\mathrm{sm}} }
\]
we get a map $X'\dashrightarrow\Im_{U_{\mathrm{sm}}/k}(\mathbf{Jac}(X'/\mathbb P^{n-1})^\vee)$ inducing a homomorphism $\mathbf{Alb}(X')\to\Im_{U_{\mathrm{sm}}/k}(\mathbf{Jac}(X'/\mathbb P^{n-1})^\vee)$. This dualizes to a homomorphism \[\xymatrix{ \Tr_{U_{\mathrm{sm}}/k}(\mathbf{Jac}(X'/\mathbb P^{n-1}))\ar[r] & \mathbf{Cl}^0(X') }. \]
Conversely,
the map (\ref{eqn:cl(X')_to_sections})
defines a morphism $\mathbf{Cl}^0(X')\times U_{\mathrm{sm}}\to\mathbf{Jac}(X'/\mathbb P^{n-1})$ of abelian $U_{\mathrm{sm}}$-schemes and hence a homomorphism of abelian varieties $\mathbf{Cl}^0(X')\to\Tr_{U_{\mathrm{sm}}/k}(\mathbf{Jac}(X'/\mathbb P^{n-1}))$ by the universal property of the Chow trace.
\end{rem}

\subsection{The connected component $\Cl^0$}\label{sec:cl^0_iso}
We first show that a stronger result holds for the subgroup of algebraically trivial divisors: the restriction map on $\Cl^0$ is an isomorphism for a general divisor.

The following lemma is elementary and presumably well-known, but we include it here for completeness. It also holds if $U$ is replaced by a quasi-finite, generically unramified cover from a normal variety.
\begin{lem}\label{lem:general_line_traceless}
Let $k$ be an algebraically closed field and $\pi\colon\mathcal A\to U\subset \mathbb P^m_k$ an abelian scheme. Then $\Tr_{U/k}(\mathcal A)\cong\Tr_{L\cap U/k}(\mathcal A|_{L\cap U})$ for a general line $L\subset\mathbb P^m_k$.
\end{lem}

\begin{proof}
For any line $L\subset\mathbb P^m_k$, we can restrict the morphism $\Tr_{U/k}(\mathcal A)\times U\to\mathcal A$ of abelian $U$-schemes to $L$, so $\Tr_{U/k}(\mathcal A)\to\Tr_{L\cap U/k}(\mathcal A|_{L\cap U})$ comes from the universal property.
For a map in the opposite direction, pick $b\in B$ general and consider the family of lines $L_\lambda\ni b$. For each $\lambda$ let $\Tr_{L_\lambda\cap U/k}(\mathcal A|_{L_\lambda \cap U})$ be the trace of the abelian scheme $\mathcal A_{L_\lambda\cap U}\to L_\lambda\cap U$. Then the fibers over $b$ of the maximal abelian $k$-subschemes of $\mathcal A|_{L_\lambda \cap U}\to L_\lambda\cap U$
\[ \{\tau_\lambda(\Tr_{L_\lambda\cap U/k}(\mathcal A|_{L_\lambda\cap U})\times (L_\lambda\cap U))_b \mid L_\lambda\ni b\}\]
form a family of abelian subvarieties of $\mathcal A_b$. So for a general $\lambda$ they must be constant \cite[\S 2]{chow55} and equal to some abelian $k$-variety $A_0$, which is isogenous to $\Tr_{L_\lambda\cap U/k}(\mathcal A|_{L_\lambda\cap U})\times (L_\lambda\cap U)$, so the traces for these $\lambda$ are isogenous. Pick one and call it $T$. Then we have $T\otimes_k k(U) \to \mathcal A_\eta$, so by the universal property of the trace we get a morphism of abelian varieties $T\to\Tr_{U/k}(\mathcal A)$.\end{proof}

\begin{thm}\label{prop:cl^0_iso}
Let $X$ be a normal projective variety of dimension $n\geq 3$ over an algebraically closed field $k$, let $\mathcal L_0$ a very ample line bundle on $X$, and let $d\geq 2$ be an integer. For a general complete intersection $Y$ of $\leq n-2$ members of $|\mathcal L_0^{\otimes d}|$, the restriction map $\Cl^0(X)\to\Cl^0(Y)$ is an isomorphism.
\end{thm}

\begin{proof}
Let $|V|$ be as in Assumption~\ref{assumption:inj}. First we consider complete intersection surfaces of members of $|V|$.
Let $L\subset\mathbb P^{n-1}$ be a general line corresponding to a complete intersection surface $S\subset X$. Then the abelian schemes $\mathbf{Jac}(X'/\mathbb P^{n-1})\to U_{\mathrm{sm}}$ and $\mathbf{Jac}(X'/\mathbb P^{n-1})|_{L\cap U_{\mathrm{sm}}}\to L\cap U_{\mathrm{sm}}$ have the same trace, and these are the same as $\Cl^0(X)$ and $\Cl^0(S)$, respectively, as noted in the remark below Corollary~\ref{cor:cl(x)_to_jac_generic}.

The isomorphism $\Cl^0(X)\cong\Cl^0(Y)$ for a general divisor $Y$ in $|V|$ follows from applying the surface case to $S\hookrightarrow Y$ and $S\hookrightarrow X$. Since $|V|\subset|\mathcal L|$ is general, then repeatedly applying the divisor case gives the result.
\end{proof}

\subsection{Injectivity for $\Cl$}\label{sec:injectivity_pf}

In the following lemma we assume $k\neq\overline{\mathbb F}_p$ to use N\'eron's specialization theorem.

\begin{lem}\label{lem:jac_generic_to_vg}
If $k\neq\overline{\mathbb F}_p$, then for $k$-points $b$ away from a field-locally thin set in $\mathbb P^{n-1}(k)$, the map
\[\xymatrix{ \mathbf{Jac}(C'_\eta)(K) \ar[r] & \mathbf{Jac}(C'_b)(k)} \] is injective.
\end{lem}

\begin{proof}
Let $A=\Tr_{U_{\mathrm{sm}}/k}(\mathbf{Jac}(X'/\mathbb P^{n-1}))$.
After possibly replacing $U_{\mathrm{sm}}$ by a smaller open subset $U$,
by \cite[Theorem 6.4]{conrad06} and Poincar\'e reducibility we can find isogenies of abelian $U$-schemes
\[\xymatrix{ (A\times U) \times \mathcal Q \ar[r] & \mathbf{Jac}(X'/\mathbb P^{n-1}) \ar[r] & (A\times U) \times \mathcal Q} \]
where $\mathcal Q=\mathbf{Jac}(X'/\mathbb P^{n-1})/(A\times U)$ is the quotient and $\Tr_{U/k}(\mathcal Q)=0$.
So it suffices to consider separately the two cases of a constant abelian scheme and one with trivial trace.

 For the traceless family $\mathcal Q\to U$, the group of rational sections is finitely generated by the Lang--N\'eron theorem \cite[Theorem 1]{lang-neron} (see also \cite[Theorem 2.1]{conrad06}). So by N\'eron's Theorem~\ref{thm:neron's_theorem_av} there are infinitely many points $b\in U_{\mathrm{sm}}(k)$ outside a field-locally thin set such that the specialization $\mathcal Q_\eta(K) \to \mathcal Q_b(k)$ is injective.
Next, the constant part has only constant sections, since a non-constant section would give a non-constant morphism $U\to A$, which cannot exist since $U$ is an open subset of $\mathbb P^{n-1}$. A priori we could have torsion in the kernel, but this does not happen by Corollary~\ref{cor:cl(x)_to_jac_generic} since $A=\Cl^0(X)$.
\end{proof}

\begin{rem}
If $k$ has infinite transcendence degree over the prime subfield, Lemma~\ref{lem:jac_generic_to_vg} can be proven for $b$ very general over $k_0$ without using the Lang--N\'eron theorem or N\'eron's specialization theorem. In this setting the group of rational sections of $\mathcal Q\to U$ is countable because each rational section gives an isolated point of the Chow variety \cite[Lemma 3.6]{graber-starr13}, so any $b$ with coordinates transcendental over a field of definition for $\mathcal Q$ is outside the images of their intersections.
\end{rem}

Putting this together with Corollary~\ref{cor:cl(x)_to_jac_generic} shows injectivity of the restriction map from $X$ to infinitely many curves. We can factor this through higher-dimensional complete intersections in $|V|$:
\begin{cor}\label{cor:inj_to_surfaces}
If $n\geq 3$ in Assumption~\ref{assumption:inj}, then without further assumptions on the ground field $k$, the restriction map $\Cl(X)\to\Cl(Y)$ is injective for a general complete intersection $Y$ of $\leq n-2$ members of $|\mathcal L|$.
\end{cor}

\begin{proof}
We first show injectivity of the restriction map for a general complete intersection surface $S$ of members of $|V|$, since this will imply injectivity of the restriction map to any divisor in $|V|$ containing $S$.
If $k=\overline{\mathbb F}_p$, let $k\subsetneq k'$ be an extension of algebraically closed fields.
Let $N$ be the set of $k'$-points for which the specialization map of Lemma~\ref{lem:jac_generic_to_vg} is injective; the complement of $N$ is contained inside a field-locally thin subset of $\mathbb P^{n-1}(k')$.
Since for any finitely-generated field $k_1\subset k'$ the set $\mathbb P^{n-1}(k_1)\setminus N(k_1)$ is a finite union of type 1 and type 2 thin sets (Definition~\ref{defn:thin_set}), and since $\mathbb P^1(k_1)$ is not thin, then for a general $k$-line $L\subset\mathbb P_k^{n-1}$, the base change $L_{k'}$ will contain a $k'$-point $b$ in $N$ and such that $b\in U_{\mathrm{sm}}(k')$.
For such $L$ and $b\in L_{k'}(k')$, the compositions of the maps in Corollary~\ref{cor:cl(x)_to_jac_generic} and Lemma~\ref{lem:jac_generic_to_vg} is injective and factors through $\Cl(S_{k'})$, where $S=\overline{\phi^{-1}(L)}$ is the complete intersection surface corresponding to the line $L$, so $\Cl(X_{k'})\to\Cl(S_{k'})$ is injective.
The base change to $k'$ does not change $\Clns$ by Lemma~\ref{lem:clns_field_extn}, and we have $\Cl^0(X)=\Cl^0(S)$ and $\Cl^0(X_{k'})=\Cl^0(S_{k'})$ by Corollary~\ref{prop:cl^0_iso}, so we conclude that over $k$ the restriction map $\Cl(X)\to\Cl(S)$ is injective.

Since $|V|\subset|\mathcal L|$ was general, this implies that the restriction map to a general divisor in $|\mathcal L|$ is injective. Repeatedly applying the divisor case gives the result for general complete intersections.
\end{proof}

\section{Surjectivity}\label{sec:surj}

In this section we prove surjectivity modulo torsion of the restriction map on class groups:

\begin{thm}[Corollaries~\ref{cor:surj_for_L^2_higher_dim}, \ref{cor:other_fields_higher_dim}, and~\ref{cor:surj_for_2A+B_surface}]\label{prop:surjectivity}
Let $X$ be a normal projective variety of dimension $\geq 3$ over an algebraically closed field $k\neq\overline{\mathbb F}_p$, and let $\mathcal L_0$ be a very ample line bundle on $X$. Then there exists $Y\in |\mathcal L_0^{\otimes d}|$ in the complement of a countable union of proper subvarieties such that the restriction map
\[\xymatrix{ \Cl(X) \ar[r] & \Cl(Y)} \]
\begin{enumerate}
\item
has torsion cokernel for $d=4$ or $\geq 6$ if $\dim X=3$,
\item
is surjective for $d\geq 2$ if $\dim X\geq 4$ and if $k$ either
\begin{enumerate}
\item
has characteristic 0 or
\item
has infinite transcendence degree over its prime subfield, and
\end{enumerate}
\item
has $p$-power torsion cokernel for $d\geq 2$ if $\dim X\geq 4$ and $\chara k=p$.
\end{enumerate}
\end{thm}

The proof of Theorem~\ref{prop:surjectivity} proceeds in several steps. In \S\ref{sec:surj_pf_infinite_tdeg}, we first show surjectivity modulo torsion when the ground field $\mathbb K$ has infinite transcendence degree over its prime subfield. When \(\dim X=3\), the first step is to show surjectivity of the restriction map to a very general reducible $S_1+S_2$, using a theorem of Graber--Starr on extending sections of abelian schemes. Then a degeneration argument shows surjectivity (modulo torsion) to a very general $T\in|\mathcal L^{\otimes 2}|$ (Proposition~\ref{prop:surj_for_L^2_to_surfaces}), and an additional degeneration is used for surjectivity (modulo torsion) for odd multiples of $\mathcal L$ (Proposition~\ref{prop:surj_for_2A+B}). When \(\dim X\geq 4\), a direct application of Graber--Starr's theorem yields the result, and no degeneration is necessary (Corollary~\ref{cor:surj_for_L^2_higher_dim}). Finally, in \S\ref{section:other_fields}, we extend the result to any field $k\neq\overline{\mathbb F}_p$.

In the case when $X=\mathbb P^3$ and $\mathcal L=\mathcal O(2)$, the argument becomes very explicit, so before proceeding with the general proof, we first illustrate the strategy of the argument in the case of quartic $K3$ surfaces.

\subsection{Quartic surfaces in $\mathbb P^3$}\label{sec:quartic_surfaces}

Let $K\not\subset\overline{\mathbb F}_p$ be a field, $C\subset\mathbb P^3_K$ a degree 4 elliptic curve of rank at least 18, and $q_1,q_2,p_1,\ldots,p_{15}\in C(K)$ points that are independent in $\Pic(C)/\langle H\rangle$, where $H$ is the restriction of the hyperplane class on $\mathbb P^3$. For each $i=1,2$, the degree 2 line bundle $2q_i$ on $C$ uniquely determines a smooth quadric surface $Q_i=(f_i=0)$ such that the image of the restriction $\Pic(Q_i)\to\Pic(C)$ is generated by $2q_i$ and $H-2q_i$. The points $p_1,\ldots,p_{15}$ determine a quartic surface $T=(g=0)$ and a unique sixteenth point $p_{16}\in C(K)$ such that $T\cap C=p_1+\cdots+p_{16}$.

\begin{prop}
In the above setting, for any $\lambda\not\in\overline{K}$, the quartic surface defined by $f_1f_2+\lambda g=0$ has Picard rank 1.\end{prop}

\begin{proof}

Let $|\Lambda|\cong\mathbb P^1$ be the pencil in $|\mathcal O_{\mathbb P^3}(4)|$ spanned by $Q_1+Q_2$ and $T$, and let 
$\tilde{X}=\{ s_0 f_1f_2+s_1 g=0 \}\subset\mathbb P^3\times{\mathbb P^1}_{[s_0\colon s_1]}$
 be the total space of the pencil.

Now consider the Chow variety $\mathrm{WDiv}(\tilde{X}/{\mathbb P^1})$ of relative Weil divisors \cite[3.21.3]{kol_moduli_book}. The components of $\mathrm{WDiv}(\tilde{X}/{\mathbb P^1})$ are defined over $\overline{K}$, so for any $\lambda\not\in\overline{K}$ the point $[1\colon\lambda]\in{\mathbb P^1}$ will not be in the image of any component of $\mathrm{WDiv}(\tilde{X}/{\mathbb P^1})$ not dominating the base.
Let $T_\lambda=(f_1f_2+\lambda g=0)$ be the corresponding quartic surface, and let $D_\lambda$ be a curve on $T_\lambda$.
After a base change by a smooth curve $\Gamma\to{\mathbb P^1}$ we can find a divisor $\mathcal D$ on $\tilde{X}\times_{\mathbb P^1}\Gamma$ whose fiber over $\tilde{\lambda}$ is $D_\lambda$, where $\tilde{\lambda}\in\Gamma(\overline{K(\lambda)})$ maps to $[1\colon\lambda]\in{\mathbb P^1}(\overline{K(\lambda)})$; see the proof of Proposition~\ref{prop:surj_for_L^2_to_surfaces} for details.

Let $\tilde{0}\in\Gamma$ be a point mapping to $[1:0]\in{\mathbb P^1}(K)$, and denote the fiber $(\tilde{X}\times_{\mathbb P^1}\Gamma)_{\tilde{0}}=\tilde{Q}_1+\tilde{Q}_2$, so that $\tilde{Q}_i\xrightarrow{\cong} Q_i$ and $\tilde{C}:=\tilde{Q}_1\cap\tilde{Q}_2\xrightarrow{\cong}C$.
By a local computation (Lemma~\ref{lem:local_computation_cl}), the class group of the local ring of $\tilde{X}\times_{\mathbb P^1}\Gamma$ at a point $\tilde{x}$ in $\tilde{C}$ is either
\begin{itemize}
\item
$\mathbb Z/r\mathbb Z$ if $\tilde{x}$ maps to a point of $C\setminus\{p_1,\ldots,p_{16}\}$ and $r$ is the ramification index of $\Gamma\to\mathbb P^1$ at $\tilde{0}$, or
\item
$\mathbb Z$, generated by $(\tilde{Q}_1)_{\tilde{x}}$ (or equivalently $(-\tilde{Q}_2)_{\tilde{x}}$), if $\tilde{x}$ maps to some $p_j$.
\end{itemize}

Let $\tilde{p}_j\in\tilde{C}$ be the point mapping to $p_j$ for $1\leq j\leq 16$. Then $\mathcal D$ uniquely determines a Weil divisor class $D^\circ_i$ on $\tilde{Q}_i$, and the restriction $D^\circ_1|_{\tilde{Q}_1\cap\tilde{Q}_2}-D^\circ_2|_{\tilde{Q}_1\cap\tilde{Q}_2}$ is supported on $\{\tilde{p}_1,\ldots,\tilde{p}_{16}\}$.
Write $D^\circ_1|_{\tilde{Q}_1\cap\tilde{Q}_2}-D^\circ_2|_{\tilde{Q}_1\cap\tilde{Q}_2}=\sum_{j=1}^{16} a_j\tilde{p}_j$.
Then the divisor $\sum_{j=1}^{16} a_j p_j$ on $C$ is in the subgroup $\langle H, 2q_1,2q_2\rangle$ of $\Pic(C)$ generated by the images of $\Pic(Q_i)\to\Pic(C)$, so for some $b_1,b_2\in\mathbb Z$ we have
\[2b_1q_1+2b_2q_1+\sum_{j=1}^{16} a_j p_j\sim 2b_1q_1+2b_2q_1+4a_{16}H+\sum_{j=1}^{15} (a_j-a_{16}) p_j\in\langle H\rangle.\]
Since $q_1,q_2,p_1,\ldots,p_{15}$ are independent in $\Pic(C)/\langle H\rangle$, this implies $-2b_1=-2b_2=a_j-a_{16}=0$ for $1\leq j\leq 15$, so
\[ D^\circ_1|_{\tilde{Q}_1\cap\tilde{Q}_2}-D^\circ_2|_{\tilde{Q}_1\cap\tilde{Q}_2} \sim a_{16}\sum_{j=1}^{16}\tilde{p}_j .\]

Then the divisor $\mathcal D-a_{16} \tilde{Q}_1$ on $\tilde{X}\times_{\mathbb P^1}\Gamma$ is $\mathbb Q$-Cartier along $\tilde{C}$.
Let $D_i$ be the divisor on $\tilde{Q}_i$ defined by $\mathcal D-a_{16} \tilde{Q}_1$. By abuse of notation we also denote by $D_i$ the corresponding divisor under $\tilde{Q}_i\xrightarrow{\cong} Q_i$. Since $\Pic(Q_1)\times_{\Pic(C)}\Pic(Q_2)\cong\mathbb Z$ is generated by the restriction of the hyperplane class on $\mathbb P^3$, there is some $d\in\mathbb Z$ such that $D_1\in|\mathcal O_{\mathbb P^3}(d)|_{Q_1}|$ and $D_2\in|\mathcal O_{\mathbb P^3}(d)|_{Q_2}|$.

Now on $\tilde{X}\times_{\mathbb P^1}\Gamma$, the difference between the pullback of $\mathcal O_{\mathbb P^3}(d)$ and $\mathcal D-a_{16} \tilde{Q}_1$ is a divisor that is Cartier on the generic fiber, $\mathbb Q$-Cartier along the fiber $\tilde{Q}_1+\tilde{Q}_2$ over $\tilde{0}\in\Gamma$, and trivial when restricted to $\tilde{Q}_1+\tilde{Q}_2$.
So by injectivity of specialization of N\'eron--Severi groups, the fiber over $\tilde{\lambda}$ is in $\Cl^0(T_\lambda)=0$. That is, $D_\lambda\in|\mathcal O_{\mathbb P^3}(d)|_{T_\lambda}|$.

\end{proof}

\subsection{Surjectivity over fields of infinite transcendence degree}\label{sec:surj_pf_infinite_tdeg}
In this section we assume the ground field $\mathbb K$ has infinite transcendence degree over its prime subfield and use the notion of very general over $k_0$ discussed in \S\ref{sec:very_general}.

We will use the following result of Graber--Starr \cite{graber-starr13} on extending sections of abelian schemes. Their result is stated for very general incident line pairs and 2-planes over an uncountable field, but their proof works for fields of infinite transcendence degree.
\begin{thm}[{\cite[Theorem 1.3]{graber-starr13}}]\label{thm:gs_extend_sections}
Let $\mathbb K$ be an algebraically closed field of infinite transcendence degree, $B$ a normal variety of dimension $\geq 2$, and $f\colon B\to \mathbb P^{\dim B}$ a generically finite, generically unramified morphism. Let $\mathcal A\to B$ be an abelian scheme, and let $k_0 \subset\mathbb K$ be a field of definition for $\mathcal A$ and $B$. Then for any pair of incident lines $L_1\cup L_2\subset\mathbb P^{\dim B}$ very general over $k_0$, the restriction map
\[\text{\{sections of }\mathcal A\to B\}\to\{\text{sections of }\mathcal A|_{f^{-1}(L_1\cup L_2)}\to f^{-1}(L_1\cup L_2)\}\] is a bijection.
The same result also holds when the very general incident line pair $L_1\cup L_2$ is replaced by a very general 2-plane $P$.
\end{thm}
Applying their result for incident line pairs to the relative Jacobian shows surjectivity for a reducible member of $|\mathcal L^{\otimes 2}|$:

\begin{cor}\label{lem:surj_on_reducible_member}
In the setting of Assumption~\ref{assumption:inj}, assume that $X$ is defined over a field $\mathbb K$ of infinite transcendence degree and that $n\geq 3$.
Let $S_1,S_2$ be a pair of complete intersection surfaces of members of $|V|$ very general over $k_0$, and let $D_i$ be a Weil divisor on each $S_i$ such that $D_1|_{S_1\cap S_2}\sim D_2|_{S_1\cap S_2}$. Then there is a Weil divisor $D$ on $X$, uniquely determined up to linear equivalence, such that $D|_{S_i}\sim D_i$ for each $i$. That is, the restriction map \[ \Cl(X)\to\Cl(S_1)\times_{\Pic(S_1\cap S_2)}\Cl(S_2) \] is an isomorphism.
\end{cor}

\begin{proof}
Let $S_i'$ denote the strict transform of $S_i$, let $L_i$ be the line such that $S'_i=\phi'^{-1}(L_i)$, and consider the fiber product $\Cl(S'_1)\times_{\Pic(S_1\cap S_2)}\Cl(S'_2)$, where the restriction maps $\Cl(S_i')\to\Cl(S_i)\to\Pic(S_1\cap S_2)$ are pushforward by $\pi|_{S'_i}$ composed with the Gysin homomorphisms \cite[Chapter 6]{ful84}.
Let $D_i$ be divisors on $S_i$ whose restrictions agree in $\Pic(S_1\cap S_2)$. The pullback $\beta|_{S_i'}^*D_i$ of $D_i$ is defined because the base points of $|V|$ are smooth points of $S_i$ by generality.

The restrictions $(\beta|_{S_i'}^*D_i)|_{C'_\eta}$ have the same degree, so by the procedure of Lemma~\ref{lem:cl^0_to_sections} $D_1$ and $D_2$ define a section of $\mathbf{Jac}(X'/\mathbb P^{n-1})\to U_{\mathrm{sm}}$ over $(L_1\cup L_2)\cap U_{\mathrm{sm}}$.
This extends uniquely to a section $\sigma$ of $\mathbf{Jac}(X'/\mathbb P^{n-1})\to U_{\mathrm{sm}}$ by Theorem~\ref{thm:gs_extend_sections} {\cite[Theorem 1.3]{graber-starr13}}.
By Lemma~\ref{lem:cl^0_to_sections} we can find a divisor $D'$ on $X'$ defining $\sigma$ and such that $(\pi_* D')|_{S_i} \sim D_i$ for each $i$.
\end{proof}

The same argument applied to the 2-plane case of Theorem~\ref{thm:gs_extend_sections} gives surjectivity for complete intersection threefolds of members of $|V|$ that are very general over $k_0$. We also get surjectivity to higher-dimensional complete intersections by factoring the restriction map.
\begin{cor}\label{cor:surj_for_L^2_higher_dim}
In the setting of Assumption~\ref{assumption:inj}, assume that $X$ is defined over a field $\mathbb K$ of infinite transcendence degree and that $n\geq 4$.
For a complete intersection $Y$ of $r\leq n-3$ members of $|\mathcal L|$ very general over $k_0$, the restriction map
\[\xymatrix{ \Cl(X) \ar[r] & \Cl(Y)} \] is surjective.
\end{cor}

\begin{proof}
It suffices to show the case when $Y\in|V|$ is a divisor. Let $Z$ be a threefold obtained as the complete intersection of $Y$ with members of $|V|$ very general over $k_0$. The restriction map $\Cl(X)\to\Cl(Z)$ factors through $\Cl(Y)\to\Cl(Z)$, which is injective by Corollary~\ref{cor:inj_to_surfaces}, so to show the claim it suffices to show surjectivity of $\Cl(X)\to\Cl(Z)$. Let $Z'\subset X'$ be the strict transform of $Z$ and $P=\phi'(Z)\subset\mathbb P^{n-1}$ the corresponding $2$-plane.

By Lemma~\ref{lem:cl^0_to_sections}, a divisor $D$ on $Z$ defines a section of $\mathbf{Jac}(X'/\mathbb P^{n-1})\to U_{\mathrm{sm}}$ over $P\cap U_{\mathrm{sm}}$. This extends uniquely to a section $\sigma$ over $U_{\mathrm{sm}}$ by Theorem~\ref{thm:gs_extend_sections}. By Lemma~\ref{lem:cl^0_to_sections} we can find a divisor $D'$ on $X'$ defining $\sigma$ and such that $(\pi_* D')|_Z \sim D$.
\end{proof}

In order to specialize from a sufficiently general member of $|\mathcal L^{\otimes 2}|$ to a reducible one in $|\mathcal L|+|\mathcal L|$, we need the following conjecture of Koll\'ar on independence of intersection points \cite[Conjecture 15]{kol_determines} (which generalizes \cite[Sublemma]{Griffiths-Harris}). Voisin proved Koll\'ar's conjecture over fields of infinite transcendence degree, and we include her proof below.
This result will ensure in Proposition~\ref{prop:surj_for_L^2_to_surfaces} that the points $p_i$ are independent in $\Pic(C)$ modulo the hyperplane section, as in the case for quartic surfaces in $\mathbb P^4$ when $C$ was a degree 4 elliptic curve in \S\ref{sec:quartic_surfaces}.
\begin{prop}[{\cite{voisin20}}]\label{lem:kol_sections_indep_conj}
Let $\mathbb K$ be an algebraically closed field of infinite transcendence degree.
Let $C$ be a smooth projective curve over $\mathbb K$ and $\mathcal M$ a very ample line bundle on $C$.
Assume that either
\begin{enumerate}[label=(\alph*)]
\item\label{case:assumption-case-1} $\chara\mathbb K=0$, or 
\item\label{case:assumption-case-2} $\mathcal M=\mathcal M_0^{\otimes d}$ for an integer $d\geq 2$ and $\mathcal M_0$ a very ample line bundle.\end{enumerate}
Let $A\subsetneq\mathbf{Pic}^0(C)$ be an abelian subvariety and $\mathcal M\in{G}\subset\Pic(C)$ a finitely-generated subgroup.
Let $k^C\subset\mathbb K$ be a field of definition for $C$, $\mathcal M$, $A$, and $G$.
For a divisor $D\in |\mathcal M|$ write $\Supp D=\bigcup_i p_i(D)$. Then the map
\[ \bigoplus_i \mathbb Z[p_i(D)] / \sum_i [p_i(D)] \hookrightarrow\Pic(C)/\langle A(\mathbb K),{G}\rangle\]
is an injection for $D\in |\mathcal M|$ very general over $k^C$.
\end{prop}

\begin{proof}[Proof of Voisin]
Let $d=\deg \mathcal M$, let $Z\subset C\times|\mathcal M|$ be the universal family, and let $U\subset|\mathcal M|$ be the open subset parametrizing divisors consisting of $d$ distinct points. Let $Z^{(d)}\to|\mathcal M|$ be the universal family of divisors in $|\mathcal M|$ with an ordering of the $d$ (not necessarily distinct) points.

First we show that $Z^{(d)}_U$ is irreducible, or equivalently that the monodromy group of $Z_U\to U$ is the full symmetric group $S_d$ (i.e. that it satisfies the ``uniform position lemma" \cite{Harris-Galois}). It suffices to show that the monodromy group is 2-transitive and contains a transposition. 2-transitivity is equivalent to irreducibility of $Z^{(2)}_U$, which can be identified with the set $\{(p_1,p_2,D) \mid p_1,p_2\in\Supp D\}\subset C\times C\times U$. The fiber over a point $(p_1,p_2)\in C\times C$ is the linear system $|\mathcal M-p_1-p_2|$, so since $\mathcal M$ is very ample the dimension of this linear system is independent of the $p_i$ and $Z^{(2)}_U\to U$ is a projective bundle and thus irreducible.
Deforming a section of $\mathcal M$ with a double zero gives a transposition; such a section exists in case~\ref{case:assumption-case-1} by taking a section simply tangent to $C$ \cite[\S4]{Harris-Galois}, and in case~\ref{case:assumption-case-2} by the assumption that $d\geq 2$.\footnote{If $\chara\mathbb K=p>0$ and $d=1$, it can happen that no such section of $\mathcal M$ with exactly one double zero exists, e.g. if every tangent point is a flex point.}

Next, for $\gamma\in{G}$ and integers $n_1,\ldots,n_d,m\in\mathbb Z$ such that $m\neq 0$ and $(n_1,\ldots,n_d)\neq(0,\ldots,0)$, define the subset
\[ Z(n_1,\ldots,n_d,m,\gamma)=\left\{ (p_1,\ldots,p_d,\sum_{i=1}^d p_i ) \in Z_U^{(d)} \mid m(\gamma-\sum_{i=1}^d n_i[p_i])\in A\right\}
\]
of $Z_U^{(d)}\subset C\times\cdots\times C\times |\mathcal M|$.
Then \[\bigcup_{(n_1,\ldots,n_d,m,\gamma)} Z(n_1,\ldots,n_d,m,\gamma) \subset Z_U^{(d)}\] is a countable union of $k^C$-subvarieties, and to show that the complement contains $\mathbb K$-points, it suffices to show that each $Z(n_1,\ldots,n_d,m,\gamma) \subsetneq Z_U^{(d)}$.

By contradiction, assume that some $Z(n_1,\ldots,n_d,m,\gamma) = Z_U^{(d)}$.
Then $m(\gamma-\sum_{i=1}^d n_i[p_i])\in A$ for every $(p_1,\ldots,p_d,\sum_{i=1}^d p_i)\in Z_U^{(d)}$, and since the monodromy group is the entire symmetric group $S_d$, this also holds for any permutation of the $p_i$.
Now the permutation representation on $\mathbb Q^d$ is the sum of two irreducible representations: the diagonal and its complement, which is spanned by differences of the basis vectors. So we have that either $n_1=\cdots =n_d$, or that $m(\gamma-n[p_1]+n[p_2])\in A$ for any $p_1,p_2\in C$. In the second case we get that $\textbf{Pic}^0(C)/A$ is torsion since the differences $[p_1]-[p_2]$ generate $\Pic^0$, which is a contradiction. So we are in the first case, i.e. a multiple of $\sum_{i=1}^d p_i$.
\end{proof}

We will also need to replace $X$ by suitable alteration in the proof of Proposition~\ref{prop:surj_for_L^2_to_surfaces} to use an argument involving specialization of $\mathbb Q$-Cartier divisors.
\begin{lem}\label{lem:partial_alteration}
Let $X$ be a normal projective threefold over an algebraically closed field $k$, $\mathcal L$ a very ample line bundle on $X$, and $|\Lambda|\subset|\mathcal L|$ a general pencil. Then there exists a morphism from a projective variety $\psi\colon X^+\to X$ such that
\begin{enumerate}
\item\label{item:partial_alteration_pi}
$\psi\colon X^+\to X$ is finite and purely inseparable over the regular locus of $X$,
\item\label{item:partial_alteration_q-factorial}
$X^+$ is $\mathbb Q$-factorial except possibly over finitely many points $q_1,\ldots, q_r$ of $X$,
\item\label{item:partial_alteration_L}
the pullback of a general member of $|\Lambda|$ to $X^+$ is smooth, and
\item\label{item:partial_alteration_L^b}
for any $b\geq 1$ the pullback of a general member of $|\mathcal L^{\otimes b}|$ is smooth along the exceptional divisors of $\psi\colon X^+\to X$. In particular, the pullback of a general member of $|\mathcal L^{\otimes b}|$ is $\mathbb Q$-factorial.
\end{enumerate}
\end{lem}

\begin{proof}
If $\chara k=0$, then any resolution of singularities $\psi\colon X^+ \to X$ works, so assume $\chara k=p>0$.
Let $\mathcal Y\to|\Lambda|$ be the total space of the pencil. The generic fiber $Y$ is a geometrically normal surface over the function field $K$ of $|\Lambda|\cong\mathbb P^1$, and after a finite purely inseparable field extension the base change $Y\otimes_K K^{1/p^e}$ admits a resolution of singularities $Y^+$ that is smooth over $K^{1/p^e}$. (This is because the base change of $Y$ to the perfect closure $K^{1/p^\infty}$ of $K$ admits a resolution of singularities that is smooth over $K^{1/p^\infty}$ \cite{Abhyankar57}, and this is all defined over some finite subextension of $K\subset K^{1/p^\infty}$.)
For a $k$-variety, we use the superscript $^{1/p^e}$ to denote the source of the $e\textsuperscript{th}$ iterate of relative Frobenius over $k$. Spreading out $Y^+$ and $Y\otimes_K K^{1/p^e}$ over an open subset $U^{1/p^e}$ of $|\Lambda|^{1/p^e}$ we get a smooth family $\mathcal Y^+\to U^{1/p^e}$ of surfaces, a family $\mathcal Y|_U\times_U U^{1/p^e}$ of normal surfaces, and a diagram
\[\xymatrix{ \mathcal Y^+ \ar[rr]^-{\text{bir}} & & \mathcal Y|_U\times_U U^{1/p^e} \ar[rr]^-{\text{univ homeo}}  \ar[d] & & \mathcal Y|_U \ar[rr]^-{\text{bir}}_-{\Bl_{\Bs|\Lambda|}} \ar[d] & & \pi_1(\mathcal Y|_U) \subset X \ar@{-->}[lld] \\
& & U^{1/p^e} \ar[rr] & & U } \]
where $\pi_1(\mathcal Y|_U)$ is a dense open subset of $X$ whose complement is the union of finitely many divisors $Y_1,\ldots,Y_l$ in $|\Lambda|$. By Bertini's theorem (after possibly removing finitely many points from $U$) the birational morphism $\mathcal Y^+\to\mathcal Y|_U\times_U U^{1/p^e}$ is an isomorphism over the regular locus of $X$.

Let $\theta\colon Z \to X$ be the normalization of $X$ in $K^{1/p^e}$. Then $\theta^{-1}(X^{\mathrm{reg}})\to X^{\mathrm{reg}}$ is a purely inseparable cover 
that agrees with $\mathcal Y^+$ over the intersection $X^{\mathrm{reg}}\cap \pi_1(\mathcal Y|_U)$, so we may glue $\theta^{-1}(X^{\mathrm{reg}})$ to $\mathcal Y^+$ to get a $\mathbb Q$-factorial variety with a morphism
\[\xymatrix{ X^+_0:=\theta^{-1}(X^{\mathrm{reg}})\cup\mathcal Y^+ \ar[r] & X^{\mathrm{reg}}\cup\pi_1(\mathcal Y|_U)}=X\setminus\{q_1,\ldots,q_r\} \]
that factors as the composition of a purely inseparable cover and a birational morphism.
The finitely many points $q_1,\ldots,q_r$ are given by the intersection of $Y_1\cup\cdots\cup Y_l$ with the singular locus of $X$, since by generality the base locus of $|\Lambda|$ is contained in the regular locus of $X$.

Let $X^+$ be a projective closure of $X^+_0$. Over $X\setminus\{q_1,\ldots,q_r\}$ the morphism $\psi\colon X^+\to X$ agrees with $X^+_0\to X$, so properties \eqref{item:partial_alteration_pi} and \eqref{item:partial_alteration_q-factorial} hold.
By construction, the pullback of any member $\neq Y_1,\ldots,Y_l$ of $|\Lambda|$ to $X^+$ is smooth except possibly over $\Bs|\Lambda|$. So the pullback of a general member of $|\Lambda|$ is smooth over the singular locus of $X$, i.e. along the exceptional divisors of $\psi$,
and by openness of smoothness this holds for a general member of $|\mathcal L|$ as well.
Applying this to sums of $b$ general members of $|\mathcal L|$ and again using openness of smoothness, the pullback of a general member of $|\mathcal L^{\otimes b}|$ is smooth along the exceptional divisors of $\psi$. In particular, by Bertini's theorem, the pullback of a general member of $|\mathcal L^{\otimes b}|$ is $\mathbb Q$-factorial.
\end{proof}

We are now ready to show surjectivity for $\mathcal L^{\otimes 2}$.
Beginning with a divisor on a very general member $T$ of $|\mathcal L^{\otimes 2}|$, we degenerate to a reducible member. By Corollary~\ref{lem:surj_on_reducible_member} the divisor on this reducible member will lift to a divisor $D^X$ on $X$, and we show that up to torsion and $\Cl^0(X)$, the restriction of $D^X$ to $T$ agrees with the original divisor.
Even when $X$ is smooth, the family is singular, so to deal with these singularities we apply Proposition~\ref{lem:kol_sections_indep_conj} to ensure that the divisors we work with are $\mathbb Q$-Cartier in some necessary places.

\begin{prop}\label{prop:surj_for_L^2_to_surfaces}
In the setting of Assumption~\ref{assumption:inj}, assume that $X$ is defined over a field $\mathbb K$ of infinite transcendence degree and that $n\geq 3$.
Let $S_1,S_2$ be complete intersection surfaces in $|V|$ that are very general over $k_0$, and pick a pencil in $|\mathcal L^{\otimes 2}|$ through $S_1+S_2$ that is very general over $\overline{k^{S_1}}$ and $\overline{k^{S_2}}$.

If $T$ is a complete intersection surface of $n-3$ members of $|V|$ and a member of $|\mathcal L^{\otimes 2}|$ and if $T$ is very general over the algebraic closure of the compositum of the fields of definition for $S_1$, $S_2$, and the pencil, then the restriction map
\[\xymatrix{ \Cl(X) \ar[r] & \Cl(T)} \]
has torsion cokernel.
\end{prop}

\begin{proof}

We may assume $n=3$ by Corollary~\ref{cor:surj_for_L^2_higher_dim}.
Let $S_1,S_2\in|V|$ be very general over $k_0$. Then each $S_i$ is a normal surface, $S_1+S_2$ is a reducible member of $|V|+|V|\subset|\mathcal L^{\otimes 2}|$, and $C:=S_1\cap S_2$ is a smooth curve contained in the regular locus of $X$.
Let $s\in|\mathcal L^{\otimes 2}|$ be the point corresponding to $S_1+S_2$, take a pencil through $s$ that is very general over $\overline{k^{S_1}}$ and $\overline{k^{S_2}}$, and let $\tilde{X}\to\mathbb P^1$ be the total space of the pencil. Then $\beta\colon\tilde{X}\to X$ is the blowup along the reducible curve given by the base locus.

The Chow variety $\mathrm{WDiv}(\tilde{X}/\mathbb P^1)$ parametrizing Weil divisors on the fibers of $\tilde{X}\to \mathbb P^1$ \cite[3.21.3]{kol_moduli_book} has countably many components not dominating the base. The image of each of these components is a point defined over the algebraic closure of the compositum of $k^{S_1}$, $k^{S_2}$, and a field of definition for the pencil. So any $\mathbb K$-point $g$ corresponding to a divisor $T=\tilde{X}_{g}$ very general over this field will not be in the image of any of these components. Let $\{p_j\}$ denote the intersection points $C\cap T$.

Let $D_{g}$ be an integral divisor on the normal surface $T$, and consider the component $\Omega$ of $\mathrm{WDiv}(\tilde{X}/\mathbb P^1)$ containing $[D_{g}]$; by assumption $\Omega$ dominates $\mathbb P^1$. Let $\Gamma\subset\Omega$ be a curve passing through the point $[D_{g}]$ and dominating the base $\mathbb P^1$. Then the base change $\mathrm{WDiv}(\tilde{X}/\mathbb P^1)\times_{\mathbb P^1}\Gamma=\mathrm{WDiv}(\tilde{X}\times_{\mathbb P^1}\Gamma/\Gamma)\to\Gamma$ admits a section whose image contains the divisor $[D_{g}]$ on the fiber $(\tilde{X}\times_{\mathbb P^1}\Gamma)_{\tilde{{g}}}=\tilde{X}_{g}$, where $\tilde{{g}}$ maps to ${g}$.
That is, we have a divisor $\mathcal D$ on $\tilde{X}\times_{\mathbb P^1}\Gamma$ whose restriction to the fiber over $\tilde{{g}}$ is $D_{g}$, and whose restriction to any fiber is a 1-cycle on the corresponding complete intersection surface.
Denote the base change $\alpha\colon\tilde{X}\times_{\mathbb P^1}\Gamma\to\tilde{X}$.

Let $\tilde{s}\in\Gamma(\mathbb K)$ map to $s$, and let $r$ be the ramification index of $\Gamma\to\mathbb P^1$ at $\tilde{s}$. The fiber $(\tilde{X}\times_{\mathbb P^1}\Gamma)_{\tilde{s}}$ has two irreducible components $\tilde{S}_1$ and $\tilde{S}_2$, with each $\tilde{S}_i$ mapping isomorphically to $S_i$. Let $\tilde{p}_j\in\tilde{S}_1\cap\tilde{S}_2$ denote the point mapping to $p_j\in C$.

The restriction of $\mathcal D_{\tilde{s}}$ to each $\tilde{S}_i$ determines a unique Weil divisor class $D^\circ_i$ in each $\Cl(\tilde{S}_i)$.
By Lemma~\ref{lem:local_computation_cl} we know that on $(\tilde{X}\times_{\mathbb P^1} \Gamma)\setminus\{\tilde{p}_j\}$ every Weil divisor is locally Cartier along $\tilde{S}_1\cap\tilde{S}_2\setminus\{\tilde{p}_j\}$ after multiplication by $r$.
So the difference $D^\circ_1|_{\tilde{S}_1\cap\tilde{S}_2}-D^\circ_2|_{\tilde{S}_1\cap\tilde{S}_2}$ is a divisor on $\tilde{S}_1\cap\tilde{S}_2$ supported on $\{\tilde{p}_j\}$.
The local computation also shows that the local class group is $\mathbb Z=\langle(\tilde{S}_1)_{\tilde{p}_j}\rangle=\langle(-\tilde{S}_2)_{\tilde{p}_j}\rangle$ at each $\tilde{p}_j$.

We now apply Proposition~\ref{lem:kol_sections_indep_conj} to the very ample line bundle $\mathcal L^{\otimes 2}|_C$ on $C$ and to the subgroups $A,G$ defined as follows. For $i=0,1$, choose a finite set of elements in $\Cl(\tilde{S}_i)$ whose images generate $\Clns(\tilde{S}_i)$; this is possible by the Theorem of the Base (Theorem~\ref{thm:structure_class_group}). Let $G\subset\Pic(C)$ be the subgroup generated by their images under the restriction maps
\[\xymatrix{ \Cl(\tilde{S}_i) \ar[r]^-{(\alpha\circ\beta)_*} & \Cl(S_i) \ar[r] & \Pic(C)}, \] together with $\mathcal L^{\otimes 2}|_C$. Let $A\subset\Pic^0(C)$ be the subgroup generated by the images of $\Cl^0(\tilde{S}_i)$ under these restrictions for $i=0,1$. Lemma~\ref{lem:cokernel_to_Jac(C)} shows that $A$ satisfies the hypotheses of Proposition~\ref{lem:kol_sections_indep_conj}, so we conclude that  $(\beta_*\alpha_*D^\circ_1)|_C-(\beta_*\alpha_*D^\circ_2)|_C=a \sum p_j \in |\mathcal L^{2a}|_C|$ for some $a\in\mathbb Z$.

We now consider instead the divisor $\mathcal D-a\tilde{S}_1$ on $\tilde{X}\times_{\mathbb P^1}\Gamma$.
On the fiber over $\tilde{g}$ this agrees with $\mathcal D_{\tilde{g}}=D_g$.
Over $\tilde{s}$,
the fiber $(\mathcal D-a\tilde{S}_1)_{\tilde{s}}$ determines a unique Weil divisor class $D_i$ in each $\Cl(\tilde{S}_i)$, and $\mathcal D-a \tilde{S}_1$ is $\mathbb Q$-Cartier along $C$ since it becomes trivial in the local class group at each $\tilde{p}_j$. So the restrictions of $\beta_*\alpha_* D_1$ and $\beta_*\alpha_* D_2$ to $C$ agree in $\Pic(C)$, and
the divisors $\beta_*\alpha_* D_1$ on $S_1$ and $\beta_*\alpha_* D_2$ on $S_2$ satisfy the hypotheses of Corollary~\ref{lem:surj_on_reducible_member}. So we get a unique Weil divisor class $D^X\in\Cl(X)$ such that $D^X|_{S_i}\sim\beta_*\alpha_* D_i$ for $i=0,1$.
We now have two divisors $D^X|_{T}$ and $D_g$ on $T$.

\noindent\textbf{Claim:} $D^X|_T- D_g$ is torsion in $\Clns(T)$.

The claim implies surjectivity of $\Cl(X)\to\Cl(T)$ up to torsion, since $\Cl^0(X)\cong\Cl^0(T)$ by Corollary~\ref{prop:cl^0_iso}. So it remains to show the claim.

Let $(\beta\circ\alpha)^*D^X$ denote the divisor $(D^X\times \Gamma)|_{\tilde{X}\times_{\mathbb P^1}\Gamma}$ on $\tilde{X}\times_{\mathbb P^1}\Gamma$.
The idea of the proof is that $(\beta\circ\alpha)^*D^X-(\mathcal D-a\tilde{S}_1)$ is a divisor on $\tilde{X}\times_{\mathbb P^1}\Gamma$ whose specialization to the reducible fiber $\tilde{s}$ is trivial. So if it is $\mathbb Q$-Cartier, then we can conclude by specialization of the N\'eron--Severi group that a multiple is algebraically trivial on the very general fiber $\tilde{g}$. In general we can reduce to the case where the divisor is $\mathbb Q$-Cartier in the necessary places.

First assume $X$ is smooth. Then the generic fiber $\tilde{X}_\eta$ of $\tilde{X}\to\mathbb P^1$ is a smooth surface over the function field of $\mathbb P^1$.
So the generic fiber $(\tilde{X}\times_{\mathbb P^1}\Gamma)_\xi$ of the base change $\tilde{X}\times_{\mathbb P^1}\Gamma\to\Gamma$ is also smooth over the function field of $\Gamma$, and the divisor $(\beta\circ\alpha)^*D^X-(\mathcal D-a\tilde{S}_1)$ on $\tilde{X}\times_{\mathbb P^1}\Gamma$ is Cartier on the generic fiber. It's also $\mathbb Q$-Cartier on the reducible fiber $\tilde{S}_1+\tilde{S}_2$ over $\tilde{s}$, since $D^X$ is Cartier on $X$ and $\mathcal D-a\tilde{S}_1$ was chosen to be $\mathbb Q$-Cartier (with Cartier index $r$) along $\tilde{S}_1\cap\tilde{S}_2$. So $r((\beta\circ\alpha)^*D^X-(\mathcal D-a\tilde{S}_1))_\xi$ is in the kernel of
\[\xymatrix{ \Pic((\tilde{X}\times_{\mathbb P^1}\Gamma)_\xi) \ar[r] & \NS((\tilde{X}\times_{\mathbb P^1}\Gamma)_\xi) \ar[r] & \NS((\tilde{X}\times_{\mathbb P^1}\Gamma)_{\overline{\xi}}) \ar[r]^-{\mathrm{sp}_{\tilde{s}}} & \NS((\tilde{X}\times_{\mathbb P^1}\Gamma)_{\tilde{s}})}.\]
The specialization map on N\'eron--Severi groups is injective up to torsion because intersection numbers specialize, so some multiple is algebraically trivial on the very general fiber. That is, $lr(D^X|_{T}-D_g)\in\Pic^0(T)$ for some $l\in\mathbb Z_{>0}$.

If $X$ is singular, let $\psi\colon X^+\to X$ be the alteration defined in Lemma~\ref{lem:partial_alteration}.
We have a divisor $(\beta\circ\alpha)^*D^X - (\mathcal D -a \tilde{S}_1)$ on $\tilde{X}\times_{\mathbb P^1}\Gamma$, and we now claim that the pullback $((\beta\circ\alpha)^*D^X - (\mathcal D -a \tilde{S}_1))\times_X X^+$ to $\tilde{X}\times_{\mathbb P^1}\Gamma\times_X X^+$ is $\mathbb Q$-Cartier near the reducible fiber $(\tilde{S}_1+\tilde{S}_2)\times_X X^+$.
For the $D^X$ part, we have that $(\beta\circ\alpha)^*D^X\times_X X^+$ is $\mathbb Q$-Cartier near $(\tilde{S}_1+\tilde{S}_2)\times_X X^+$ because $X^+$ is $\mathbb Q$-factorial except over $\{q_j\}$, and $S_1,S_2$ do not contain any of the points $q_j$ by generality.
Next, the divisor $\mathcal D -a \tilde{S}_1$ on $\tilde{X}\times_{\mathbb P^1}\Gamma$ is $\mathbb Q$-Cartier along $\tilde{S}_1\cap\tilde{S}_2$ by construction, and away from $S_1\cap S_2$ the base change $(S_1+S_2)\times_X X^+\to S_1+S_2$ becomes $\mathbb Q$-factorial by property~\eqref{item:partial_alteration_L^b}. So $\mathcal D -a \tilde{S}_1$ is also $\mathbb Q$-Cartier after base change.

$((\beta\circ\alpha)^*D^X - (\mathcal D -a \tilde{S}_1))\times_X X^+$ is also $\mathbb Q$-Cartier on the geometric generic fiber $(\tilde{X}\times_{\mathbb P^1}\Gamma\times_X X^+)_{\overline{\xi}}$ of $\tilde{X}\times_{\mathbb P^1}\Gamma\times_X X^+ \to \Gamma$.
The exceptional divisors of $(\tilde{S}_1+\tilde{S}_2)\times_X X^+\to \tilde{S}_1+\tilde{S}_2$ are specializations of the exceptional divisors of $(\tilde{X}\times_{\mathbb P^1}\Gamma\times_X X^+)_{\overline{\xi}} \to (\tilde{X}\times_{\mathbb P^1}\Gamma)_{\overline{\xi}}$, so there is some (Cartier) divisor $E$ with exceptional support on the geometric generic fiber such that $r(((\beta\circ\alpha)^*D^X - (\mathcal D -a \tilde{S}_1))\times_X X^+)_{\overline{\xi}}+E$ becomes trivial after specializing to the fiber $(\tilde{S}_1+\tilde{S}_2)\times_X X^+$ over $\tilde{s}$.
Then $r(((\beta\circ\alpha)^*D^X - (\mathcal D -a \tilde{S}_1))\times_X X^+)_{\overline{\xi}}+E$ is numerically trivial, so some multiple $lr (((\beta\circ\alpha)^*D^X - (\mathcal D -a \tilde{S}_1))\times_X X^+ +E)|_{\tilde{T}_g\times_X X^+}$ is algebraically trivial on the very general fiber. Pushing forward we get $\beta_*\alpha_*(lr((\beta\circ\alpha)^*D^X - (\mathcal D -a \tilde{S}_1))|_{\tilde{T}_g}) = D^X|_{T}-D_g \in \Cl^0(T)$.
\end{proof}

\begin{lem}\label{lem:local_computation_cl}
In the proof of Proposition~\ref{prop:surj_for_L^2_to_surfaces},
$\tilde{X}\times_{\mathbb P^1} \Gamma$ is locally $\mathbb Q$-factorial over $C\setminus\{p_i\}$ with Cartier index equal to the ramification index $r$ of $\Gamma\to\mathbb P^1$ at $\tilde{s}$.
\end{lem}

\begin{proof}

Note that this computation applies to any collection of Cartier divisors $S_1\in|\mathcal L_1|, S_2\in|\mathcal L_2|$, and $T\in |\mathcal L_1\otimes\mathcal L_2|$ as long as each one is smooth at $x$.

Let $x\in C=S_1\cap S_2\subset X$ be a closed point. Then $x$ is in the regular locus of $X$, $S_1$, and $S_2$, so we may choose local parameters $y_1,y_2,y_3$ at $x$ such that $S_i$ is locally defined by $y_i$ at $x$. Let $q$ be a local equation for $T$ at $x$; then $x\in T$ if and only if $q$ is in the maximal ideal of $k[[y_1,y_2,y_3]]$, i.e. not a unit. The total space of the pencil is locally $\widehat{\mathcal O}_{\tilde{X},x}\cong k[[y_1,y_2,y_3,s]]/(y_1y_2-sq)$ where $s$ is a local parameter for $\mathcal O_{\mathbb P^1,s}$.

Let $t$ be a local parameter at $\mathcal O_{\Gamma,\tilde{s}}$. Then at a point $\tilde{x}$ in $\tilde{X}\times_{\mathbb P^1}\Gamma$ mapping to $x$ \[ \widehat{\mathcal O}_{\tilde{X}\times_{\mathbb P^1}\Gamma,\tilde{x}}\cong k[[y_1,y_2,y_3,t]]/(y_1y_2-ut^rq) \] for some unit $u$. If $x\not\in T$, then $\widehat{\mathcal O}_{\tilde{X}\times_{\mathbb P^1} \Gamma,\tilde{x}}\cong k[[y_1,y_2,y_3,t]]/((u'q)^{-1}y_1y_2-t^r)$ is an $A_{r-1}$-singularity, so its class group is $\mathbb Z/r\mathbb Z$.

If $x\in T$ then $q$ is in the maximal ideal of $k[[y_1,y_2,y_3]]$. Since $T\in|\mathcal L^{\otimes 2}|$ is very general, $x$ is not in its singular locus and $Y$ not cut out by any of the same equations as $S_1+S_2$, so we may assume that our coordinates were chosen so that $q=y_3$. Then $\widehat{\mathcal O}_{\tilde{X}\times_{\mathbb P^1} \Gamma,\tilde{x}}\cong k[[y_1,y_2,y_3,t]]/(y_1y_2-ut^ry_3)\cong k[[y_1,y_2,y_n,t]]/(y_1y_2-(uy_3)t^r)$. This ring has trivial Picard group and its class group is $\mathbb Z=\langle (y_1=t=0)\rangle=\langle -(y_2=t=0)\rangle$.
\end{proof}

Over uncountable fields, Noether--Lefschetz for even multiples $\geq 4$ immediately follows from Theorem~\ref{prop:injectivity} and Proposition~\ref{prop:surj_for_L^2_to_surfaces}: if $n=3$ then the restriction map $\Cl(X)\to\Cl(T)$ is injective and is surjective up to torsion for surfaces $T$ outside of a countable union of closed subvarieties of $|\mathcal L^{\otimes 2}|$.

To get surjectivity for odd multiples in dimension 3, we consider a different type of specialization. Instead of degenerating to a reducible member of $|\mathcal L|+|\mathcal L|$, we consider a reducible member of $|\mathcal L_A^{\otimes 2}|+|\mathcal L_B|$ and apply Proposition~\ref{prop:surj_for_L^2_to_surfaces} to lift divisors from the $|\mathcal L_A^{\otimes 2}|$ component. To conclude with the same specialization argument as before, we must ensure that our divisors are $\mathbb Q$-Cartier on the reducible member (possibly after base change by an alteration of $X$), so we assume $\mathcal L_B$ satisfies the conditions of Lemmas~\ref{lem:bertini_normal} and \ref{lem:cokernel_to_Jac(C)}. This assumption is why the $d=5$ case will be missing from Corollary~\ref{cor:surj_for_2A+B_surface} below.

\begin{prop}\label{prop:surj_for_2A+B}
Let $X$ be a normal projective threefold over an algebraically closed field $\mathbb K$ of infinite transcendence degree, and let $k_0\subset\mathbb K$ be a field of definition for $X$.
Let $\mathcal L_A=\mathcal L_{A,0}^{\otimes d_A}$ where $\mathcal L_{A,0}$ is a very ample line bundle and $d_A\geq 2$ is an integer. Let $\mathcal L_B$ be a very ample line bundle such that $\mathcal L_B=\mathcal L_{B,0}^{\otimes d_B}$ for $d_B\geq 2$ and $\mathcal L_{B,0}$ ample and basepoint-free.
Assume that one of the following holds:
\begin{enumerate}[label=(\alph*)]
\item
$X$ is smooth,
\item
$\chara\mathbb K=0$, or
\item
$\mathcal L_A^{\otimes 2}\otimes\mathcal L_B \cong \mathcal M^{\otimes d}$ for some very ample line bundle $\mathcal M$ and integer $d$.
\end{enumerate}
Let $S'_1, S'_2\in|\mathcal L_A|$ be very general over $k_0$, pick a pencil in $|\mathcal L_A^{\otimes 2}|$ through $S'_1+S'_2$ that is very general over $\overline{k^{S'_1}}$ and $\overline{k^{S'_2}}$, and let $S_{2A}\in |\mathcal L_A^{\otimes 2}|$ and $S_B\in|\mathcal L_B|$ be very general over the algebraic closure of the compositum of the fields of definition for $S'_1,S'_2$, and the pencil.
Pick a pencil $|\Lambda|\subset|\mathcal L_A^{\otimes 2}\otimes\mathcal L_B|$ through $S_{2A}+S_B$ very general over the algebraic closure of the compositum of $k^{S_{2A}}$ and $k^{S_B}$.

If $T\in|\mathcal L_A^{\otimes 2}\otimes\mathcal L_B|$ is very general over the algebraic closure of the compositum of the above fields, then the restriction map
\[\xymatrix{ \Cl(X) \ar[r] & \Cl(T)} \]
is surjective up to torsion.
\end{prop}

\begin{proof}
Let $C=S_{2A}\cap S_B$, and let $\tilde{X}$ be the total space of the pencil $|\Lambda|$. Let $s$ and $g$ denote the points on $|\Lambda|=\mathbb P^1$ corresponding to $S_{2A}+S_B$ and $T$, respectively. By our choice of $T$, $g$ is not in the image of any component of the Chow variety $\mathrm{WDiv}(\tilde{X}/\mathbb P^1)$ not dominating the base.

Let $D_g$ be an integral divisor on $T$. As in the proof of Proposition~\ref{prop:surj_for_L^2_to_surfaces}, $\mathrm{WDiv}(\tilde{X}/\mathbb P^1)\times_{\mathbb P^1}\Gamma\to\Gamma$ admits a section after taking some finite base change $\Gamma\to\mathbb P^1$. So we have a divisor $\mathcal D$ on $\tilde{X}\times_{\mathbb P^1}\Gamma$ that restricts to $D_g$ on the fiber over a point $\tilde{g}\in\Gamma$ mapping to $g$.

Let $\tilde{s}\in\Gamma$ be a point mapping to $s$.
Then $(\tilde{X}\times_{\mathbb P^1}\Gamma)_{\tilde{s}}\xrightarrow{\cong} \tilde{X}_s \cong S_{2A}+S_B$, and we denote the components of the reducible fiber $(\tilde{X}\times_{\mathbb P^1}\Gamma)_{\tilde{s}}$ by $\tilde{S}_{2A}$ and $\tilde{S}_B$.
As in the proof of Proposition~\ref{prop:surj_for_L^2_to_surfaces}, there is an integer $a$ such that $\mathcal D+a\tilde{S}_B$ is $\mathbb Q$-Cartier along $\tilde{S}_{2A}\cap\tilde{S}_B$, obtained by applying Proposition~\ref{lem:kol_sections_indep_conj} to the very ample line bundle $(\mathcal L_A^{\otimes 2}\otimes\mathcal L_B)|_C$ on the curve $C$ and to the subgroups $A\subset\Pic^0(C)$ and ${G}\subset\Pic(C)$ generated by the images of $\Cl^0$ and by $(\mathcal L_A^{\otimes 2}\otimes\mathcal L_B)|_C$ and the images of a set of representatives of generators of $\Clns$, respectively, under the restrictions $\Cl(S_{2A})\to\Pic(C)$ and $\Cl(S_B)\to\Pic(C)$. We use the assumptions on $\mathcal L_B$ here to apply Lemma~\ref{lem:cokernel_to_Jac(C)} and ensure that $A\subsetneq\Pic^0(C)$.

The divisor $\mathcal D+a\tilde{S}_B$ on the family defines a unique Weil divisor class $D_{2A}\in\Cl(\tilde{S}_{2A})$.
By Theorem~\ref{prop:injectivity} and Proposition~\ref{prop:surj_for_L^2_to_surfaces} the restriction $\Cl(X)\to\Cl(S_{2A})$ is an injection with torsion cokernel, so for some $c\in\mathbb Z_{\geq 1}$ there is a unique divisor class $D^X$ on $X$ such that $D^X|_{S_{2A}}\sim c\beta_*\alpha_* D_{2A}$.
By Theorem~\ref{prop:injectivity} its restriction to $S_B$ agrees with the divisor defined by $\mathcal D+a\tilde{S}_B$.
Let $(\beta\circ\alpha)^*D^X$ denote $(D^X\times\Gamma)|_{\tilde{X}\times_{\mathbb P^1}\Gamma}$.

The divisor $(\beta\circ\alpha)^*D^X-c(\mathcal D+a\tilde{S}_B)$ on $\tilde{X}\times_{\mathbb P^1}\Gamma$ is trivial on the reducible fiber $\tilde{S}_{2A}+\tilde{S}_B$.
If $X$ is smooth or if $\chara\mathbb K=0$, then the same specialization argument in the proof of Proposition~\ref{prop:surj_for_L^2_to_surfaces} implies a multiple of this divisor is algebraically trivial when restricted to the fiber over $\tilde{g}$, and using the same argument we conclude that $\Cl(X)\to\Cl(T)$ is surjective up to torsion.

If $X$ is singular and $\chara\mathbb K=p>0$, construct a purely inseparable alteration $X^+\to X$ as in Lemma~\ref{lem:partial_alteration} using a general pencil in the very ample linear system $|\mathcal M|$.
Then $((S_{2A}+S_B)\setminus(S_{2A}\cap S_B))\times_X X^+$ is $\mathbb Q$-factorial, and since $(\beta\circ\alpha)^*D^X$ and $\mathcal D+a\tilde{S}_B$ are $\mathbb Q$-Cartier along $\tilde{S}_{2A}\cap\tilde{S}_B$ by generality and by choice of $a$, respectively, the argument of Proposition~\ref{prop:surj_for_L^2_to_surfaces} goes through. So we again conclude surjectivity of $\Cl(X)\to\Cl(T)$ up to torsion.
\end{proof}

\subsection{Surjectivity over fields $\neq\overline{\mathbb F}_p$}\label{section:other_fields}

Following a suggestion of Bjorn Poonen,
we will now apply a result of Andr\'e in characteristic 0 and Ambrosi and Christensen in characteristic $p$ to show that if surjectivity holds over fields of infinite transcendence degree, then it also holds for any field $\neq\overline{\mathbb F}_p$. The result we will use is for the Picard rank of smooth families, so we will apply it on a resolution of singularities or a suitable alteration. We first recall their results:

\begin{thm}[{\cites[Th\'eor\`eme 5.2]{andre83}[Corollary 1.7.1.5]{ambrosi18}[Theorem 1.0.1]{christensen18}}]\label{thm:andre-ambrosi-christensen}
Let $k\neq\overline{\mathbb F}_p$ be an algebraically closed field, $B$ a $k$-scheme of finite type, and $\mathcal T\to B$ a smooth morphism. Let $\eta$ be a generic point of $B$ and $\mathcal T_{\overline{\eta}}$ the corresponding geometric generic fiber. Then there exists $b\in B(k)$ such that $\rho(\mathcal T_{\overline{\eta}})=\rho(T_b)$.
\end{thm}
Using \cite[Proposition 3.6]{mp12}, this means that for such $b$ the specialization map $\NS(\mathcal T_{\overline{\eta}})\to\NS(T_b)$ is an isomorphism if $\chara k=0$ and an isomorphism up to $p$-power torsion if $\chara k=p>0$.
In the case of a family of surfaces, isomorphism up to torsion can be argued directly as follows.

\begin{cor}\label{cor:andre_for_surfaces}
Let $k\neq\overline{\mathbb F}_p$ be an algebraically closed field, $B$ a variety over $k$ with generic point $\eta$, and $\mathcal T\to B$ a family of normal surfaces. Then there exists $b\in B(k)$ such that the specialization map $\Clns(\mathcal T_{\overline{\eta}})\to\Clns(T_b)$ is an isomorphism up to torsion.
\end{cor}

\begin{proof}
First, we note that after base change by a purely inseparable dominant morphism $B'\to B$, the family $\mathcal T\times_B B'$ admits a simultaneous resolution of singularities $\tilde{\mathcal T}\to\mathcal T\times_B B'$ such that the exceptional divisors on the closed fibers specialize from the geometric generic fiber. This is obtained by spreading out a resolution of generic fiber $T\to\Spec K$ of $\mathcal T\to B$ after a purely inseparable base change: as in the argument of Lemma~\ref{lem:partial_alteration}, after some purely inseparable extension, there is a resolution of singularities of the surface $T\otimes_K K^{1/p^e}$ that is smooth over $K^{1/p^e}$. 

For a resolution of singularities of a surface $\pi\colon \tilde{T}\to T$, the rank of $\Clns(T)$ is equal to $\rho(\tilde{T})-\#\{$exceptional curves of $\pi\}$, so by construction it suffices to show the statement for the smooth family $\tilde{\mathcal T}\to B'$.

If $D^i$ are divisors on $\tilde{\mathcal T}_{\overline{\eta}}$ specializing to $D^i_b$ on the fiber $\tilde{T}_b$ over $b$, then the rank of the intersection matrix $(D^i_b\cdot D^j_b)$ is independent of $b\in B'$ \cite[Example 20.3.6]{ful84}, so $\NS(\tilde{\mathcal T}_{\overline{\eta}})\to\NS(\tilde{T}_b)$ has torsion kernel for any $b$. Thus the specialization map will be surjective up to torsion if and only if $\rho(\tilde{\mathcal T}_{\overline{\eta}})=\rho(\tilde{T}_b)$, and Theorem~\ref{thm:andre-ambrosi-christensen} produces such a $k$-point $b$.

\end{proof}

\begin{prop}\label{prop:other_fields}
Let $X$ be a normal projective variety of dimension $n\geq 3$ defined over an algebraically closed field $k\neq\overline{\mathbb F}_p$ and $\mathcal L$ a very ample line bundle on $X$.
\begin{enumerate}
\item\label{item:other_fields_higher_dim}
If $\mathcal L=\mathcal L_0^{\otimes d}$ for a very ample line bundle $\mathcal L_0$ and integer $d\geq 2$, then there is a complete intersection surface $T_b$ of $|\mathcal L^{\otimes 2}|$ with $n-3$ members of $|\mathcal L|$ defined over $k$ and such that the restriction map
\[\xymatrix{ \Cl(X) \ar[r] & \Cl(T_b)} \] has torsion cokernel.
\item\label{item:other_fields_threefold}
If $n=3$, further assume that $\mathcal L=\mathcal L_{A,0}^{\otimes 2 d_A}\otimes\mathcal L_{B,0}^{\otimes d_B}$ for very ample line bundles $\mathcal L_{A,0}, \mathcal L_{B,0}$ and integers $d_A, d_B\geq 2$ and that one of the following holds:
\begin{enumerate}[label=(\alph*)]
\item
$X$ is smooth,
\item
$\chara k=0$, or
\item
$\mathcal L \cong \mathcal M^{\otimes d}$ for some very ample line bundle $\mathcal M$ and integer $d$.
\end{enumerate}
Then there is a divisor $T_b\in|\mathcal L|$ defined over $k$ and such that the restriction map
\[\xymatrix{ \Cl(X) \ar[r] & \Cl(T_b)} \] has torsion cokernel.
\end{enumerate}
\end{prop}

\begin{proof}
Let $\mathcal T\subset X\times B$ denote the corresponding universal family of surfaces and $\mathcal T_\eta$ the generic fiber of the projection onto $S$. Note that the geometric generic fiber $\mathcal T_{\overline{\eta}}$ is normal by Theorem~\ref{lem:bertini_va}.
Let $\mathbb K$ be an algebraically closed field of infinite transcendence degree containing $k$. Let $\mathcal L_{\mathbb K}$ denote the pullback to $X_{\mathbb K}:=X\otimes_k\mathbb K$, $\mathcal T_{\mathbb K}\to B_{\mathbb K}$ the corresponding universal family, and $(\mathcal T_{\mathbb K})_{\overline{\eta}}$ the geometric generic fiber, which agrees with the base change of $\mathcal T_\eta$ to the algebraic closure of the function field of $B_\mathbb K$.

By Corollary~\ref{cor:andre_for_surfaces} there is a $k$-point $b\in B$ such that the specialization $\Clns(\mathcal T_{\overline{\eta}})\to\Clns(T_b)$ is an isomorphism up to torsion, where $T_b$ is the corresponding surface. We have the following commutative diagram.

\[\xymatrixcolsep{5pc}\xymatrix{
\Clns(X_\mathbb K) \ar[r] \ar@{=}[d]^-{\text{Lemma~\ref{lem:clns_field_extn}}} & \Clns((\mathcal T_{\mathbb K})_{\overline{\eta}}) \ar@{=}[r]^-{\text{Lemma~\ref{lem:clns_field_extn}}} &\Clns(\mathcal T_{\overline{\eta}}) \ar@{=}[d]^-{\substack{\text{after }\otimes\mathbb Q \\ \text{Corollary~\ref{cor:andre_for_surfaces}}}} \\
\Clns(X) \ar[rr] & & \Clns(T_b) } \]
Moreover, $\Clns(X_\mathbb K)\to \Clns((\mathcal T_{\mathbb K})_{\overline{\eta}})$ is injective with torsion cokernel by Theorem~\ref{prop:injectivity}, Proposition~\ref{prop:cl^0_iso}, Proposition~\ref{prop:surj_for_L^2_to_surfaces} in case~\eqref{item:other_fields_higher_dim}, and Proposition~\ref{prop:surj_for_2A+B} in case~\eqref{item:other_fields_threefold}.
Therefore $\Clns(X)\to\Clns(T_b)$ has torsion cokernel, and using Proposition~\ref{prop:cl^0_iso} again we conclude that $\Cl(X)\to\Cl(T_b)$ has torsion cokernel.
\end{proof}

In higher dimensions, one can apply the above argument to Corollary~\ref{cor:surj_for_L^2_higher_dim} to show:
\begin{cor}\label{cor:other_fields_higher_dim}
Let $X$ be a normal projective variety of dimension $n\geq 4$ defined over an algebraically closed field $k\neq\overline{\mathbb F}_p$, $\mathcal L_0$ a very ample line bundle on $X$, and $d\geq 2$ an integer. Then there is a divisor $Y_b$ in $|\mathcal L_0^{\otimes d}|$ defined over $k$ for which the restriction map
\[\xymatrix{ \Cl(X) \ar[r] & \Cl(Y_b)} \]
\begin{enumerate}
\item
is surjective if $\chara k=0$, and
\item
has $p$-power torsion cokernel if $\chara k=p>0$.
\end{enumerate}
\end{cor}
The argument is essentially the same as in Proposition~\ref{prop:other_fields}, but in positive characteristic, instead of the resolution of singularities in Corollary~\ref{cor:andre_for_surfaces} we take a smooth $p$-alteration \cite{temkin17} of the generic fiber after a purely inseparable base change. The finite part of the alteration only contributes $p$-power torsion \cite[Example 1.7.4]{ful84}; this and the torsion in the cokernel of the specialization map \cite[Proposition 3.6]{mp12} are the reasons for the torsion in the characteristic $p$ statement.

When $\mathcal L_{A,0}=\mathcal L_{B,0}$ in part~\eqref{item:other_fields_threefold} of Proposition~\ref{prop:other_fields} (and using part~\eqref{item:other_fields_higher_dim} for the $d=4$ case) we obtain
\begin{cor}\label{cor:surj_for_2A+B_surface}
Let $X$ be a normal projective threefold over an algebraically closed field $k\neq\overline{\mathbb F}_p$ and $\mathcal L_0$ a very ample line bundle on $X$. Then for any integer $d\in\{4\}\cup\mathbb Z_{\geq 6}$
the restriction map
\[\xymatrix{ \Cl(X) \ar[r] & \Cl(T)} \]
is surjective up to torsion for very general $T\in|\mathcal L_0^{\otimes d}|$.
\end{cor}
In particular we recover the statement modulo torsion for $X=\mathbb P^3_{\mathbb C}$ except for the case of degree 5 surfaces.

\bibliographystyle{alpha}
\bibliography{references_nl}

\begin{thebibliography}{{Amb}18}

\bibitem[Abh57]{Abhyankar57}
Shreeram Abhyankar.
\newblock On the field of definition of a nonsingular birational transform of
  an algebraic surface.
\newblock {\em Ann. of Math. (2)}, 65:268--281, 1957.

\bibitem[{Amb}18]{ambrosi18}
Emiliano {Ambrosi}.
\newblock {Specialization of N{\'e}ron-Severi groups in positive
  characteristic}.
\newblock {\em arXiv e-prints}, page arXiv:1810.06481, October 2018.
\newblock To appear in Ann. Sci. \'Ec. Norm. Sup\'er.

\bibitem[And96]{andre83}
Yves Andr\'{e}.
\newblock Pour une th\'{e}orie inconditionnelle des motifs.
\newblock {\em Inst. Hautes \'{E}tudes Sci. Publ. Math.}, 83:5--49, 1996.

\bibitem[BG12]{bruzzo-grassi12}
Ugo Bruzzo and Antonella Grassi.
\newblock Picard group of hypersurfaces in toric 3-folds.
\newblock {\em Internat. J. Math.}, 23(2):1250028, 14, 2012.

\bibitem[BG18]{bruzzo-grassi18}
Ugo Bruzzo and Antonella Grassi.
\newblock The {N}oether--{L}efschetz locus of surfaces in toric threefolds.
\newblock {\em Commun. Contemp. Math.}, 20(5):1750070, 20, 2018.

\bibitem[BGL20]{bruzzo-grassi-lopez20}
Ugo Bruzzo, Antonella Grassi, and Angelo~Felice Lopez.
\newblock {Existence and Density of General Components of the
  Noether--Lefschetz Locus on Normal Threefolds}.
\newblock {\em International Mathematics Research Notices}, 02 2020.

\bibitem[BLR90]{blr}
Siegfried Bosch, Werner L\"{u}tkebohmert, and Michel Raynaud.
\newblock {\em N\'{e}ron models}, volume~21 of {\em Ergebnisse der Mathematik
  und ihrer Grenzgebiete (3) [Results in Mathematics and Related Areas (3)]}.
\newblock Springer-Verlag, Berlin, 1990.

\bibitem[BN16]{brevik-nollet16}
John Brevik and Scott Nollet.
\newblock Grothendieck-{L}efschetz theorem with base locus.
\newblock {\em Israel J. Math.}, 212(1):107--122, 2016.

\bibitem[BN20]{brevik-nollet20}
John Brevik and Scott Nollet.
\newblock Moishezon's theorem and degeneration.
\newblock {\em J. Algebra}, 544:463--482, 2020.

\bibitem[Cas94]{castelnuovo1894}
Guido Castelnuovo.
\newblock Sulle superfici algebriche che ammettono un sistema doppiamente
  infinito di sezioni piane irriducibili.
\newblock {\em Rendic. Acad. Lincei}, 3:22--25, 1894.

\bibitem[CGGH83]{cggh83}
James Carlson, Mark Green, Phillip Griffiths, and Joe Harris.
\newblock Infinitesimal variations of {H}odge structure. {I}.
\newblock {\em Compositio Math.}, 50(2-3):109--205, 1983.

\bibitem[Cha13]{charles13}
Fran\c{c}ois Charles.
\newblock The {T}ate conjecture for {$K3$} surfaces over finite fields.
\newblock {\em Invent. Math.}, 194(1):119--145, 2013.

\bibitem[Cho52]{chow52}
Wei-Liang Chow.
\newblock On {P}icard varieties.
\newblock {\em Amer. J. Math.}, 74:895--909, 1952.

\bibitem[Cho55]{chow55}
Wei-Liang Chow.
\newblock Abelian varieties over function fields.
\newblock {\em Trans. Amer. Math. Soc.}, 78:253--275, 1955.

\bibitem[{Chr}18]{christensen18}
Atticus {Christensen}.
\newblock {Specialization of N{\'e}ron-Severi groups in characteristic $p$}.
\newblock {\em arXiv e-prints}, page arXiv:1810.06550, October 2018.

\bibitem[Con06]{conrad06}
Brian Conrad.
\newblock Chow's ${K}/k$-image and ${K}/k$-trace, and the {L}ang--{N}\'eron
  theorem.
\newblock {\em Enseign. Math}, 2006.

\bibitem[DK73]{sga7ii}
P.~Deligne and N.~Katz.
\newblock {\em Groupes de monodromie en g\'{e}om\'{e}trie alg\'{e}brique.
  {II}}.
\newblock Lecture Notes in Mathematics, Vol. 340. Springer-Verlag, Berlin-New
  York, 1973.
\newblock S\'{e}minaire de G\'{e}om\'{e}trie Alg\'{e}brique du Bois-Marie
  1967--1969 (SGA 7 II), Dirig\'{e} par P. Deligne et N. Katz.

\bibitem[Ein85]{ein1985}
Lawrence Ein.
\newblock An analogue of {M}ax {N}oether's theorem.
\newblock {\em Duke Math. J.}, 52(3):689--706, 1985.

\bibitem[FJ08]{fm08}
Michael~D. Fried and Moshe Jarden.
\newblock {\em Field arithmetic}, volume~11 of {\em Ergebnisse der Mathematik
  und ihrer Grenzgebiete. 3. Folge. A Series of Modern Surveys in Mathematics
  [Results in Mathematics and Related Areas. 3rd Series. A Series of Modern
  Surveys in Mathematics]}.
\newblock Springer-Verlag, Berlin, third edition, 2008.
\newblock Revised by Jarden.

\bibitem[FOV99]{fov99}
H.~Flenner, L.~O'Carroll, and W.~Vogel.
\newblock {\em Joins and intersections}.
\newblock Springer Monographs in Mathematics. Springer-Verlag, Berlin, 1999.

\bibitem[Ful98]{ful84}
William Fulton.
\newblock {\em Intersection theory}, volume~2 of {\em Ergebnisse der Mathematik
  und ihrer Grenzgebiete. 3. Folge. A Series of Modern Surveys in Mathematics
  [Results in Mathematics and Related Areas. 3rd Series. A Series of Modern
  Surveys in Mathematics]}.
\newblock Springer-Verlag, Berlin, second edition, 1998.

\bibitem[GH85]{Griffiths-Harris}
Phillip Griffiths and Joe Harris.
\newblock On the {N}oether-{L}efschetz theorem and some remarks on
  codimension-two cycles.
\newblock {\em Math. Ann.}, 271(1):31--51, 1985.

\bibitem[GK19]{gk19}
Mainak {Ghosh} and Amalendu {Krishna}.
\newblock {Bertini theorems revisited}.
\newblock {\em arXiv e-prints}, page arXiv:1912.09076, December 2019.
\newblock To appear in J. London Math. Soc.

\bibitem[Gre84]{green84}
Mark~L. Green.
\newblock Koszul cohomology and the geometry of projective varieties.
\newblock {\em J. Differential Geom.}, 19(1):125--171, 1984.

\bibitem[Gro05]{sga2}
Alexander Grothendieck.
\newblock {\em Cohomologie locale des faisceaux coh\'{e}rents et
  th\'{e}or\`emes de {L}efschetz locaux et globaux ({SGA} 2)}, volume~4 of {\em
  Documents Math\'{e}matiques (Paris) [Mathematical Documents (Paris)]}.
\newblock Soci\'{e}t\'{e} Math\'{e}matique de France, Paris, 2005.
\newblock S\'{e}minaire de G\'{e}om\'{e}trie Alg\'{e}brique du Bois Marie,
  1962, Augment\'{e} d'un expos\'{e} de Mich\`ele Raynaud. [With an expos\'{e}
  by Mich\`ele Raynaud], With a preface and edited by Yves Laszlo, Revised
  reprint of the 1968 French original.

\bibitem[GS13]{graber-starr13}
Tom Graber and Jason~Michael Starr.
\newblock Restriction of sections for families of abelian varieties.
\newblock In {\em A celebration of algebraic geometry}, volume~18 of {\em Clay
  Math. Proc.}, pages 311--327. Amer. Math. Soc., Providence, RI, 2013.

\bibitem[Har79]{Harris-Galois}
Joe Harris.
\newblock Galois groups of enumerative problems.
\newblock {\em Duke Math. J.}, 46(4):685--724, 1979.

\bibitem[Igu52]{igusa52}
Jun-ichi Igusa.
\newblock On the {P}icard varieties attached to algebraic varieties.
\newblock {\em Amer. J. Math.}, 74:1--22, 1952.

\bibitem[Jos95]{joshi95}
Kirti Joshi.
\newblock A {N}oether-{L}efschetz theorem and applications.
\newblock {\em J. Algebraic Geom.}, 4(1):105--135, 1995.

\bibitem[Kle05]{kleiman_fga_explained}
Steven~L. Kleiman.
\newblock The {P}icard scheme.
\newblock In {\em Fundamental algebraic geometry}, volume 123 of {\em Math.
  Surveys Monogr.}, pages 235--321. Amer. Math. Soc., Providence, RI, 2005.

\bibitem[{Kol}17]{kol_moduli_book}
J{\'a}nos {Koll{\'a}r}.
\newblock Families of varieties of general type.
\newblock
  \url{https://web.math.princeton.edu/~kollar/book/modbook20170720.pdf}, 2017.

\bibitem[{Kol}18]{kol_mumford}
J{\'a}nos {Koll{\'a}r}.
\newblock {Mumford divisors}.
\newblock {\em arXiv e-prints}, page arXiv:1803.07596, March 2018.

\bibitem[{Kol}20]{kol_determines}
J{\'a}nos {Koll{\'a}r}.
\newblock {What determines a variety?}
\newblock {\em arXiv e-prints}, page arXiv:2002.12424, February 2020.

\bibitem[Lef21]{lefschetz21}
Solomon Lefschetz.
\newblock On certain numerical invariants of algebraic varieties with
  application to abelian varieties.
\newblock {\em Trans. Amer. Math. Soc.}, 22(3):327--406, 1921.

\bibitem[LN59]{lang-neron}
S.~Lang and A.~N\'{e}ron.
\newblock Rational points of abelian varieties over function fields.
\newblock {\em Amer. J. Math.}, 81:95--118, 1959.

\bibitem[Mat52]{matsusaka-picard-1}
Teruhisa Matsusaka.
\newblock On the algebraic construction of the {P}icard variety.
\newblock {\em Proc. Japan Acad.}, 28:5--8, 1952.

\bibitem[Moi67]{moishezon67}
B.~G. Moishezon.
\newblock Algebraic homology classes on algebraic varieties.
\newblock {\em Izv. Akad. Nauk SSSR Ser. Mat.}, 31:225--268, 1967.

\bibitem[MP12]{mp12}
Davesh Maulik and Bjorn Poonen.
\newblock N\'{e}ron-{S}everi groups under specialization.
\newblock {\em Duke Math. J.}, 161(11):2167--2206, 2012.

\bibitem[N{\'{e}}r52]{neron52}
Andr\'{e} N{\'{e}}ron.
\newblock Probl\`emes arithm\'{e}tiques et g\'{e}om\'{e}triques rattach\'{e}s
  \`a la notion de rang d'une courbe alg\'{e}brique dans un corps.
\newblock {\em Bull. Soc. Math. France}, 80:101--166, 1952.

\bibitem[Noe82]{noether82}
M.~Noether.
\newblock Zur {G}rundlegung der {T}heorie der algebraischen {R}aumcurven.
\newblock {\em J. Reine Angew. Math.}, 93:271--318, 1882.

\bibitem[NS52]{neron-samuel52}
Andr\'{e} N\'{e}ron and Pierre Samuel.
\newblock La vari\'{e}t\'{e} de {P}icard d'une vari\'{e}t\'{e} normale.
\newblock {\em Ann. Inst. Fourier (Grenoble)}, 4:1--30 (1954), 1952.

\bibitem[RS06]{rs06}
G.~V. Ravindra and V.~Srinivas.
\newblock The {G}rothendieck-{L}efschetz theorem for normal projective
  varieties.
\newblock {\em J. Algebraic Geom.}, 15(3):563--590, 2006.

\bibitem[RS09]{rs09}
G.~V. Ravindra and V.~Srinivas.
\newblock The {N}oether-{L}efschetz theorem for the divisor class group.
\newblock {\em J. Algebra}, 322(9):3373--3391, 2009.

\bibitem[Ser92]{serre92}
Jean-Pierre Serre.
\newblock {\em Topics in {G}alois theory}, volume~1 of {\em Research Notes in
  Mathematics}.
\newblock Jones and Bartlett Publishers, Boston, MA, 1992.
\newblock Lecture notes prepared by Henri Damon [Henri Darmon], With a foreword
  by Darmon and the author.

\bibitem[Tem17]{temkin17}
Michael Temkin.
\newblock Tame distillation and desingularization by {$p$}-alterations.
\newblock {\em Ann. of Math. (2)}, 186(1):97--126, 2017.

\bibitem[Voi88]{voisin88}
Claire Voisin.
\newblock Une pr\'{e}cision concernant le th\'{e}or\`eme de {N}oether.
\newblock {\em Math. Ann.}, 280(4):605--611, 1988.

\bibitem[Voi89]{voisin89}
Claire Voisin.
\newblock Composantes de petite codimension du lieu de {N}oether-{L}efschetz.
\newblock {\em Comment. Math. Helv.}, 64(4):515--526, 1989.

\bibitem[Voi03]{voisin03}
Claire Voisin.
\newblock {\em Hodge theory and complex algebraic geometry. {II}}, volume~77 of
  {\em Cambridge Studies in Advanced Mathematics}.
\newblock Cambridge University Press, Cambridge, 2003.
\newblock Translated from the French by Leila Schneps.

\bibitem[Voi20]{voisin20}
Claire Voisin.
\newblock {Letter to J\'anos Koll\'ar}, 2020.

\bibitem[Wei46]{weil-foundations}
Andr\'{e} Weil.
\newblock {\em Foundations of {A}lgebraic {G}eometry}.
\newblock American Mathematical Society Colloquium Publications, vol. 29.
  American Mathematical Society, New York, 1946.

\bibitem[Wei50]{weil1950}
Andr\'{e} Weil.
\newblock Vari\'{e}t\'{e}s ab\'{e}liennes.
\newblock In {\em Alg\`ebre et th\'{e}orie des nombres. {C}olloques
  {I}nternationaux du {C}entre {N}ational de la {R}echerche {S}cientifique, no.
  24}, pages 125--127. Centre National de la Recherche Scientifique, Paris,
  1950.

\bibitem[Wei54]{weil54}
Andr\'{e} Weil.
\newblock Sur les crit\`eres d'\'{e}quivalence en g\'{e}om\'{e}trie
  alg\'{e}brique.
\newblock {\em Math. Ann.}, 128:95--127, 1954.

\end{thebibliography}

\end{document}